\documentclass[11pt,reqno]{amsart}
\usepackage[utf8]{inputenc}
\usepackage{enumitem}
\usepackage{amssymb}
\usepackage{mathtools}
\usepackage{pdfsync}
\usepackage[english]{babel}
\usepackage[pdftex,pagebackref,colorlinks=true,urlcolor=blue,linkcolor=blue,citecolor=red]{hyperref}
\usepackage[capitalize]{cleveref}
\usepackage{fullpage}
\usepackage{color}
\usepackage{amsmath}
\usepackage{amsfonts}
\usepackage{mathrsfs}
\usepackage{t1enc, graphicx}
\usepackage{verbatim}
\usepackage{bbm}
\usepackage[colorinlistoftodos]{todonotes}
\usepackage{enumitem}
\usepackage{bm}
\usepackage{upgreek}
\usepackage[mathcal]{eucal}
\usepackage{tikz-cd}
\usepackage{adjustbox}
\usepackage{marginnote}
\usepackage[nocompress,noadjust]{cite}
\usepackage{dsfont}

\newcommand{\R}{\mathbb{R}}

\newcommand{\Z}{\mathbb{Z}}
\newcommand{\N}{\mathbb{N}}

\renewcommand{\P}{\bold{P}}
\newcommand{\nnorm}[1]{\lvert\!|\!| #1|\!|\!\rvert}

\newcommand{\RP}{\bold{RP}}
\newcommand{\NRP}{\bold{NRP}}

\newcommand{\AP}{\bold{AP}}

\newcommand{\bQ}{\bm{Q}}

\newcommand{\bx}{{\boldsymbol x}}
\newcommand{\by}{{\boldsymbol y}}
\newcommand{\bg}{{\boldsymbol g}}

\newcommand{\bz}{{\boldsymbol z}}
\newcommand{\bt}{{\boldsymbol t}}
\newcommand{\bv}{{\boldsymbol v}}

\newcommand{\cF}{{\mathcal F}}

\newcommand{\cN}{{\mathcal N}}

\newcommand{\cZ}{{\mathcal Z}}
\newcommand{\cX}{{\mathcal X}}
\newcommand{\cY}{{\mathcal Y}}

\def\scr{\mathscr }

\def\sF{{\mathscr F}}

\def\GG{{\mathfrak G}}

\newcommand{\cltau}[1]{\operatorname{cl}_{\uptau}(#1)}
\newcommand{\cltauF}[2]{\operatorname{cl}_{\uptau^{[#1]}}(#2)}

\newcommand{\taulim}{\xrightarrow{\uptau}}
\newcommand{\taulimF}[1]{\xrightarrow{\uptau^{[#1]}}}

\newcommand{\E}{\mathop{\mathbb{E}}}
\newtheorem{theorem}{Theorem}[section]
\newtheorem{proposition}[theorem]{Proposition}
\newtheorem{conjecture}[theorem]{Conjecture}
\newtheorem*{conjecture*}{Conjecture}

\newtheorem{lemma}[theorem]{Lemma}

\newtheorem{corollary}[theorem]{Corollary}
\newtheorem*{corollary*}{Corollary}
\newtheorem*{theorem*}{Theorem}
\theoremstyle{definition}
\newtheorem{definition}[theorem]{Definition}
\newtheorem*{definition*}{Definition}

\theoremstyle{remark}

\newtheorem{remark}[theorem]{Remark}



\title{On higher order regionally proximal relations and topological characteristic factors for group actions}

\date{}

\author{Axel \'Alvarez}
\address[Axel \'Alvarez]{Departamento de Ingenier\'{\i}a Matem\'atica, Universidad de Chile, Beauchef 851, Santiago, Chile} \email{aalvarez@dim.uchile.cl}

\begin{document}



\maketitle
\begin{abstract} 

We study several aspects of higher-order regionally proximal relations for group actions. First, we develop an algebraic approach to study higher-order regionally proximal relations. To this end, we introduce a new topology on a subgroup of the universal minimal system, which can be seen as a higher-order analogue of the classical $\uptau$-topology. Using this topology, we obtain an algebraic characterization of the relation $\RP^{[d]}$ for abelian actions. Then, we study higher-order regionally proximal relations via recurrence sets, extending results of Huang, Shao, and Ye for $\mathbb{Z}$-actions to more general group actions under suitable assumptions. We then study topological characteristic factors and prove, modulo almost one-to-one factors, that the maximal factor of order $d-1$ is the topological characteristic factor of order $d$ for cubic configurations for arbitrary group actions, and for arithmetic progressions for finitely generated abelian group actions. As a consequence, we show that $\RP^{[d]}$ and $\AP^{[d]}$ coincide on minimal points for finitely generated abelian group actions, and we apply this to obtain results on independence along arithmetic progressions.

\end{abstract}

\tableofcontents

\section{Introduction}

A topological dynamical system is a pair $(X, T)$, where $X$ is a compact Hausdorff space and $T$ is a discrete group acting on $X$. In recent years, the class of nilsystems and their inverse limits has been a central object of study in topological dynamics. These systems relate to many deep dynamical properties, such as being characteristic for multiple ergodic averages, and have important applications in number theory and additive combinatorics (see, for example, \cite[Chapter 1]{Host_Kra_nilpotent_structures_ergodic_theory:2018} and the references therein). In the topological category, a particularly useful and important aspect of nilsystems is that they can be characterized by closed invariant relations, namely, the regionally proximal relations. 
In the context of $\Z$-actions, the regionally proximal relations were introduced by Host, Kra, and Maass in \cite{Host_Kra_Maass_nilstructure:2010}. In this context, they proved that these relations (more precisely, the property of these relations being trivial) are exactly what allow us to precisely characterize nilsystems and their inverse limits, and they may be viewed as higher-order analogues of the classical regionally proximal relation $\RP$. They showed, among other things, that in a distal minimal system, $\RP^{[d]}$ is an equivalence relation. Soon after, Shao and Ye in \cite{Shao_Ye_regionally_prox_orderd:2012} extended these results by proving that $\RP^{[d]}$ is an equivalence relation in any minimal system for abelian actions and that the quotient by it corresponds to the maximal pronilsystem factor. In \cite{Glasner_Gutman_Ye_higher_regionallyproximal_general_groups:2018}, Glasner, Gutman, and Ye gave a natural generalization of the definition of $\RP^{[d]}$ for general group actions and introduced the nilpotent regionally proximal relation, denoted by $\NRP^{[d]}$.

A central problem in this theory is to give characterizations of the regionally proximal relations. Having a precise description of these relations is very useful, as their description can be used to obtain applications to saturation, structure and recurrence theorems. In the case where $d=1$, this has long been a classical topic in topological dynamics (see for instance \cite{Veech_equicontinuous_structure_abelian_groups:1968,Auslander_Greschonig_Nagar_reflections_equicontinuity:2014,McMahon_relativized_weak_disjointness:1978,Veech_almost_automorphic_groups:1965,Auslander_McMahon_VanDerWoude_Sun_weak_disjoint_equi_structure_relation:1984}). More recently, extensive work has been done in the context of the higher-order versions $\RP^{[d]}$ (see for instance \cite{Alvarez_Donoso_cube_struct_univ_nil_applications:2025,Shao_Ye_regionally_prox_orderd:2012,Qiu_Ye_veech_higher_order:2025,Host_Kra_Maass_nilstructure:2010,Host_Maass_nil_ordre_deux:2007,Qiu_Zhao_topnilpotent_enveloping_nil:2022,Radic_sumsets_uniform_sets:2026}). In this paper, we focus on two types of characterizations: an algebraic one and one based on recurrence properties. We then derive several consequences.

\subsection{Characterization through algebraic properties} 
In \cite{Alvarez_Donoso_cube_struct_univ_nil_applications:2025}, the authors obtained a characterization of $\RP^{[d]}$ in algebraic terms for $\Z$-actions. However, that proof relies essentially on the characterization of systems of order $d$ via the enveloping semigroup proved in \cite{Donoso_enveloping_systems_orderd:2014,Qiu_Zhao_topnilpotent_enveloping_nil:2022}, namely, that a minimal system is of order $d$ if and only if its enveloping semigroup is a $d$-step nilpotent group. As up to now, this theorem is only known for $\Z$ actions, the argument does not extend to more general settings. This leads naturally to the question of whether one can obtain an algebraic characterization of $\RP^{[d]}$ without using the enveloping-semigroup characterization of systems of order $d$. In this paper, we develop a framework for the study of $\RP^{[d]}$ for abelian actions using the so-called algebraic theory, or Galois theory, of minimal systems proposed by Ellis and developed by several authors \cite{Ellis_Glasner_Shapiro_PI-flows:1975,Ellis_lectures_topological_dynamics:1969,Glasner_quasifactor_minimal_systems:2000, Ellis_the_Veech_struct_thm:1974, Glasner_topological_ergodic_descomposition:1994, Furstenberg_structure_distal_flows:1963,Veech_point-distal_flows:1970}. A central tool in this theory is the $\uptau$-topology, introduced by Furstenberg in \cite{Furstenberg_structure_distal_flows:1963}, which has proven useful in the study of regionally proximal relations (see, for instance, \cite{Alvarez_Donoso_cube_struct_univ_nil_applications:2025,Shao_Ye_regionally_prox_orderd:2012,Auslander_Guerin_regio_prox_and_prolongation:1997, Auslander_Ellis_Ellis_Regionally_proximal:1995,Ellis_Keynes_characterizarion_equi_relation:1971,Veech_topological_dynamics:1977,Ellis_Glasner_Shapiro_PI-flows:1975}). However, when dealing with higher-order regionally proximal relations, the $\uptau$-topology seems to be insufficient to describe the relations $\RP^{[d]}$, as it fails to capture the behavior of so-called face transformations. To overcome this, we introduce a new topology, denoted $\uptau^{[d]}$, inspired by the $\uptau$-topology. This new topology is tailored to capture the behavior of face transformations and shares many important properties with the $\uptau$-topology. As a consequence, it provides a natural algebraic framework for the study of higher-order proximal relations and allows us to extend the algebraic machinery of Ellis and Furstenberg to the setting of $\RP^{[d]}$. Using the framework developed in this paper, we give an algebraic characterization of $\RP^{[d]}$ for abelian actions, providing an alternative algebraic characterization to the one given in \cite[Theorem 5.10]{Alvarez_Donoso_cube_struct_univ_nil_applications:2025} for $\Z$-actions. We refer to \cref{sec: uptau_d_topology} and \cref{sec:univ_system} for the precise definitions of the algebraic objects.

\begin{theorem*}
    Let $(X,T)$ be a minimal system with $T$ being an abelian group. Then\begin{align*}
        \RP^{[d]}(X) =\{(x,vhx):x\in X,h\in H_{\uptau^{[d]}}(G), v\in J(M)\}.
    \end{align*}
\end{theorem*}

This theorem shows that the relation $\RP^{[d]}$ can be described in terms of the $\uptau^{[d]}$-topology. Consequently, the structure of systems of order $d$ may also be studied through this topology.

We expect that this algebraic characterization holds for systems satisfying the Bronstein condition, that is, the set of minimal points is dense in the product space. However, we are currently unable to prove this result due to technical difficulties. More specific comments on these issues are provided at the end of \cref{subsec: algebraic_RPd}.

We also conjecture the following characterization.

\begin{conjecture*}
    Let $(X,T)$ be a minimal system and let $d\geq 1$ be an integer. Then,\begin{align*}
        \mathcal{A}(\RP^{[d]}(X)) = \{(x,vhgx):x\in X, h\in H_{\uptau^{[d]}}(G), g\in D,v\in J(M)\},
    \end{align*}

    where $\mathcal{A}(\RP^{[d]}(X))$ denotes the smallest closed invariant equivalence relation containing $\RP^{[d]}(X)$.
\end{conjecture*}

It is known that this conjecture is true for $d=1$. This follows from the results in \cite{Auslander_Glasner_distal_order:2002,Veech_topological_dynamics:1977}. We show that this conjecture holds whenever it holds for distal systems. We also obtain some consequences, assuming that the conjecture is true. In particular, we give a necessary and sufficient condition for $\RP^{[d]}$ to be an equivalence relation.

We believe this new topology might be useful in studying other problems related to the $\RP^{[d]}$ relation and systems of order $d$. 

\subsection{Characterization through recurrence properties}
In \cite{Huang_Shao_Ye_nilbohr_automorphy:2016}, the authors show several characterizations of $\RP^{[d]}$ for $\Z$-actions using recurrence sets. We extend their arguments to general group actions under certain conditions. In particular, we obtain a characterization of $\NRP^{[d]}$ for finitely generated group actions, as well as necessary conditions for a pair to belong to $\RP^{[d]}$ in the case of finitely generated abelian group actions (\cref{thm: RPd_recurrence_sets}). To obtain these results, we also study multiple correlation sequences for actions of finitely generated abelian groups. The study of multiple correlations has been a central topic in ergodic theory since Furstenberg’s proof of Szemerédi’s theorem. In \cite{Bergelson_Host_Kra05}, linear correlation sequences were shown to admit a decomposition into a nilsequence and a null sequence (see \cref{subsec: correlation} for precise definitions). Since then, decompositions of multiple correlation sequences have been extensively studied in recent years (see for instance \cite{Leibman15, Leibman10,Leng_multi_correlation:2025,Donoso_Ferre_Koutsogiannis_Sun_multicorr_joint_erg:2024,Frantzikinakis15b,Le_nilsequences_multiple_correlations:2020}). In particular, we derive the following decomposition for a certain class of multiple correlation sequences.

\begin{theorem*}
    Let $(X,\cX,\mu,T)$ be an ergodic measure preserving system, with $T$ being a finitely generated abelian group, let $f\in L^{\infty}(\mu)$ and let $d\geq 1$ be an integer. The sequence $c_{f}(t)$ is the sum of a null-sequence and a $d$-step nilsequence, where

    \begin{align*}
    c_{f}(t) = \int f(x) \cdot f(tx)\cdot \ldots \cdot f(t^{d}x) \, d\mu(x).
    \end{align*}
\end{theorem*}

\subsection{Topological characteristic factors and applications}
In the study of multiple ergodic averages, characteristic factors play a crucial role. The notion of characteristic factors was first introduced in a paper by Furstenberg and Weiss \cite{Furstenberg_Weiss_ergodic_thm_double:1996}, and its relevance was solidified in the groundbreaking work of Host and Kra \cite{Host_Kra_nonconventional_averages_nilmanifolds:2005}. A counterpart of the notion of characteristic factors in a topological dynamical system was first studied by Glasner \cite{Glasner_topological_ergodic_descomposition:1994}. Later, Cai and Shao defined the notion of characteristic factors along cubes \cite{Cai_Shao_Topological_characteristic_cubes:2019}.

In \cref{subsec: charactersitic_cubic}, we show that the system of order $d-1$ is the topological characteristic factor along cubes of order $d$, modulo almost one-to-one factors, for any group action, thereby extending and providing a new proof of the results in \cite{Alvarez_Donoso_cube_struct_univ_nil_applications:2025,Cai_Shao_Topological_characteristic_cubes:2019}. Specifically, we prove the following theorem.

\begin{theorem*}
    Let $(X, T)$ be a minimal metric system and  $d\geq 2$ be an integer. Let $\pi:(X, T)\to (X_{d-1}, T)$ be the factor to the maximal factor of order $d-1$. Then, there is a commutative diagram of factors of minimal systems

    \[\begin{tikzcd}
	X && {X'} \\
	\\
	{X_{d-1}} && {X_{d-1}'}
	\arrow["\pi"', from=1-1, to=3-1]
	\arrow["{\theta'}"', from=1-3, to=1-1]
	\arrow["{\pi'}", from=1-3, to=3-3]
	\arrow["\theta", from=3-3, to=3-1]
    \end{tikzcd}\]

    such that $(X_{d-1}',T)$ is the topological characteristic factor along cubes of order $d$ of $(X',T)$, where $\theta,\theta'$ are almost one-to-one factors.
\end{theorem*}

In \cite{Glasner_Huang_Shao_Weiss_Ye_Topological_characteristic_factors:2020}, the authors significantly advanced Glasner’s earlier work on topological characteristic factors by showing that, for $\Z$-actions, for any $d\geq 2$, the maximal factor of order $\infty$ of a minimal system is a topological characteristic factor along arithmetic progressions of order $d$, modulo almost one-to-one factors. This result was later refined in \cite{Ye_Yu_polynomial_saturation:2025}, where it was shown that, for any $d\geq 2$, the factor of order $d-1$ is a topological characteristic factor of order $d$, for any $d\geq 2$, again modulo an almost one-to-one factor. The breakthrough result of \cite{Glasner_Huang_Shao_Weiss_Ye_Topological_characteristic_factors:2020} has inspired extensive research on topological characteristic factors and their applications in recent years (see for instance \cite{Alvarez_Donoso_cube_struct_univ_nil_applications:2025,Ye_Yu_polynomial_saturation:2025,Huang_Shao_Ye_top_induced_poly_comb:2023,Glasscock_sim_approx_nil_thick:2024,Glasscock_Koutsogiannis_Le_Moreira_Richter_Robertson_structure_polynomial_return:2025,Qiu_poly_orbits_tot_minimal:2023,Qiu_Yu_saturated_cubes_measure:2023,Shao_Xu_saturation_R_flows:2025,Wu_Xu_Ye_Structure_saturated:2023,Qiu_Xu_Ye_Yu_saturation_product:2025})

Using the connection between $\RP^{[d]}$ and recurrence sets, we extend the main results of \cite{Glasner_Huang_Shao_Weiss_Ye_Topological_characteristic_factors:2020,Ye_Yu_polynomial_saturation:2025} from $\Z$-actions to finitely generated abelian group actions. Our approach is based on a key idea from \cite{Glasner_Huang_Shao_Weiss_Ye_Topological_characteristic_factors:2020}, but unlike that work, our proof does not use the saturation theorem (\cite[Theorem 3.2]{Glasner_Huang_Shao_Weiss_Ye_Topological_characteristic_factors:2020}). Specifically, we prove the following theorem.

\begin{theorem*}
    Let $(X, T)$ be a minimal metric system, with $T$ being a finitely generated abelian group, and $d\geq 2$ be an integer. Let $\pi:(X, T)\to (X_{d-1}, T)$ be the factor to the maximal factor of order $d-1$. Then, there is a commutative diagram of factors of minimal systems

    \[\begin{tikzcd}
	X && {X'} \\
	\\
	{X_{d-1}} && {X_{d-1}'}
	\arrow["\pi"', from=1-1, to=3-1]
	\arrow["{\theta'}"', from=1-3, to=1-1]
	\arrow["{\pi'}", from=1-3, to=3-3]
	\arrow["\theta", from=3-3, to=3-1]
    \end{tikzcd}\]

    such that $(X_{d-1}',T)$ is the topological characteristic factor of order $d$ of $(X',T)$, where $\theta,\theta'$ are almost one-to-one factors.
\end{theorem*}

In \cite{Glasner_Huang_Shao_Ye_regionally_arithmetic_prog_nil:2020}, the relation $\AP^{[d]}$ is introduced for $\Z$-actions. We define the natural generalization of this relation for arbitrary group actions. As a consequence of the result of topological characteristic factors, we obtain a characterization of the minimal points of $\RP^{[d]}$ and $\AP^{[d]}$, which is new even for $\Z$-actions.

 \begin{theorem*}
     Let $(X,T)$ be a minimal system, with $T$ being a finitely generated abelian group, and let $d\geq 1$ be an integer. If $(x,y)\in X\times X$ is a minimal point, then the following statements are equivalent:\begin{enumerate}[label=(\arabic*)]
         \item $(x,y)\in \RP^{[d]}(X)$.
         \item $\{x,y\}^{d+2} \subseteq N_{d+2}(X)$.
         \item $(x,y)\in \AP^{[d]}(X)$.
     \end{enumerate}
 \end{theorem*}

 One consequence of this characterization of the minimal points of $\RP^{[d]}$ is that $\RP^{[d]}$ is the smallest closed invariant equivalence relation containing $\AP^{[d]}$.

 Another application of this characterization of the minimal points of $\RP^{[d]}$ and topological characteristic factors is to study independence pairs along arithmetic progressions. The notion of {\em independence} was first introduced in \cite{Kerr_Li_independence_top_C*_dynamics:2007}. It corresponds to a modification of the notion of interpolator studied in \cite{Glasner_Weiss_quasifactors_zeroentropy:1995,Huang_Ye_local_variational:2006}. For recent results on independence, see \cite{Dong_Donoso_Maass_Shao_Ye_infinite_step_nil:2013,Qiu_independence_automorphy_higher_order:2023,Cai_Shao_top_charact_independence:2022,Kerr_Li_independence_top_C*_dynamics:2007,Huang_Li_Ye_family_indep_top_meas:2012,Qiu_Yu_saturated_cubes_measure:2023}. In \cite{Glasner_Huang_Shao_Ye_regionally_arithmetic_prog_nil:2020}, the authors define independence pairs along arithmetic progressions, which are also studied in \cite{Cai_Shao_top_charact_independence:2022}. As a consequence of our results, we obtain that every minimal system without nontrivial independence pairs along arithmetic progressions is an almost one-to-one extension of its maximal factor of order $\infty$ for finitely generated abelian group actions (\cref{thm: non_Indap}).

\subsection*{Acknowledgments}

The author is grateful to Sebastián Donoso for his guidance during the preparation of this article and for his helpful comments. The author was supported by ANID-Subdirección de Capital Humano/Doctorado Nacional/2025-21251865 and Centro de Modelamiento Matemático (CMM) FB210005, BASAL funds for centers of excellence from ANID-Chile.

\section{Background}

\subsection{Dynamical systems}

\subsubsection{Topological dynamical systems}

A topological dynamical system (or just a system) is a pair $(X, T)$, where $X$ is a compact Hausdorff space and $T$ is a discrete group acting as a group of homeomorphisms of the space $X$. When $X$ is a metric space, we say that $(X,T)$ is a {\em metric system}. If $(X_{i},T)_{i\in I}$ is a family of systems, the action of $T$ on the product space $\prod_{i\in I} X_{i}$ is defined coordinatewise: $t(x_{i})_{i\in I} = (tx_{i})_{i\in I}$ for each $t\in T$ and $(x_{i})_{i\in I} \in\prod_{i\in I} X_{i}$.

The system is {\em transitive} if there is a point $x\in X$ such that its orbit $Tx:=\{tx: t\in T\}$ is dense in $X$. The system is {\em minimal} if the orbit of any point is dense in $X$. A pair of points $(x,y)$ in $X\times X$ is {\em proximal} if there is a net $(t_{\lambda})_{\lambda\in\Lambda}$ in $T$ and a point $z\in X$ such that $\lim t_{\lambda}x = \lim t_{\lambda}y=z$ and $(x,y)$ is a {\em distal} pair if it is not proximal. A point $x\in X$ is a {\em distal point} if it is proximal only to itself. The set of proximal pairs is denoted by $\P(X)$, and called the {\em proximal relation}.

An important class of topological dynamical systems is the class of equicontinuous systems. A system $(X, T)$ is {\em equicontinuous} if the collection of maps defined by the group $T$ is an equicontinuous family. These systems can be characterized by the {\em regionally proximal relation}. A pair $(x,y)\in X\times X$ is said to be regionally proximal if there are nets $(x_{i})_{i\in I},(y_{i})_{i\in I}$ in $X$ and $(t_{i})_{i\in I}$ in $T$ with $x_{i}\to x$, $y_{i}\to y$ and $(t_{i}x_{i},t_{i}y_{i})\to (z,z)$, for some $z\in X$. The set of regionally proximal pairs is denoted by $\RP(X)$ and is the {\em regionally proximal relation}.

A continuous onto map $\pi\colon (X, T) \to (Y, T)$ between systems is a {\em factor} (or an {\em extension}) if $t\pi(x)=\pi(tx)$ for every $x\in X$ and $t\in T$. We say that the factor is proximal (distal) if, for every  $y \in Y$, every pair of points in $\pi^{-1}(y)$ is proximal (distal). We say that a factor $\pi$ between metric systems is {\em almost one-to-one} if there exists a dense $G_{\delta}$ set $X_{0} \subseteq X$ such that $\pi^{-1}(\pi(x))=\{x\}$ for any $x\in X_{0}$. If a factor $\pi\colon (X, T) \to (Y, T)$ between minimal systems is an almost one-to-one factor, then it is also a proximal factor ({\cite[Lemma 2.1]{Donoso_Durand_Maass_Petite_automorphism_low_complexity:2016}}).

One can relativize the notion of equicontinuity. A factor map $\pi\colon (X, T)\to (Y, T)$ is called equicontinuous if for every $\alpha\in \mathcal{U}_{X}$, there exists $\beta \in \mathcal{U}_{X}$ such that if $(x,y)\in R_{\pi}$ and $x,y\in \beta$ then $(tx,ty)\in \alpha$ for all $t\in T$. Here, $R_{\pi} = \{(x,y)\in X\times X:\pi(x)=\pi(y)\}$ and $\mathcal{U}_{X}$ denotes the uniform structure on $X$ (see {\cite[Appendix II]{Auslander_minimal_flows_and_extensions:1988}} for basic properties of uniform spaces). This class of factors can be characterized by a relativized regionally proximal relation. Let $R$ be a closed invariant equivalence relation in $X$ and $\pi: X\to X/R$ the quotient map, we define\begin{align*}
    Q(R)=\bigcap_{\alpha\in \mathcal{U}_{X}} \overline{T(R\cap \alpha)}.
\end{align*}

When $R=R_{\pi}$, we call $Q(R_{\pi})$ the {\em $\pi$-regionally proximal relation}. $Q(R_{\pi})$ is trivial if and only if $\pi$ is equicontinuous.

A particular example of an equicontinuous factor is a group factor. We say that $\pi\colon (X, T)\to (Y, T)$ is a {\em group factor with group $K$} whenever the following conditions are fulfilled:\begin{enumerate}[label=(\roman*)]
            \item $K$ is a compact Hausdorff topological group
            
            \item There is a continuous mapping $(x,k)\mapsto xk: X\times K\to X$ such that\begin{enumerate}
                \item $\forall x\in X,\forall k_{1},k_{2}\in K$: $x(k_{1}k_{2})=(xk_{1})k_{2}$, $xe_{k}=x$.
                \item $\forall t\in T,\forall k\in K, \forall x\in X$: $t(xk)=(tx)k$.
            \end{enumerate}

            \item $\forall x\in X$: $\pi^{-1}(\pi(x))= xK$.
        \end{enumerate} 

\subsubsection{Measure preserving systems} 

A {\em measure preserving system} is a quadruple $(X,\cX,\mu,T)$, where $T$ is a topological group and $(X,\cX,\mu)$ is a Lebesgue probability space such that $\mu$ is $T$-invariant, that is, $\mu(tA)=\mu(A)$ for every $A\in \cX$ and every $t\in T$. We write $\mathcal{I}(T)$ for the $\sigma$-algebra $\{A\in \cX: tA=A\text{ for all }t\in T\}$ of invariant sets. The system is called {\em ergodic} if every $T$-invariant set has measure either $0$ or $1$.

For a measure preserving system $(X,\cX,\mu,T)$, a {\em measurable factor} is used with two meanings: it is a $T$-invariant sub-$\sigma$-algebra $\cY$ of $\cX$ or a system $(Y,\cY,\nu,T)$ and a measurable map $\pi:X\to Y$ such that $\pi\mu=\nu$ and $t\circ \pi=\pi\circ t$ for every $t\in T$. These two definitions coincide under the identification of the $\sigma$-algebra $\cY$ of $Y$ with the invariant sub-$\sigma$-algebra $\pi^{-1}(\cY)$ of $\cX$. In this case, we say that $Y$ is a measurable factor of $X$. If $f$ is an integrable function on $X$, we denote by $\E(f\mid\cY)$ the conditional expectation of $f$ on the measurable factor $\cY$. We write $\E(f\mid Y)$ for the function on $Y$ defined by $\E(f\mid\cY)=\E(f\mid Y)\circ \pi$. This conditional expectation is characterized by\begin{align*}
    \int_{X} f\cdot g \circ \pi \,d\mu = \int_{Y} \E(f\mid Y)\cdot g\, d\nu
\end{align*}

for all $g\in L^{\infty}(\nu)$.

\subsection{Host-Kra cube groups associated with a group}

Let $d \geq 1$ be an integer, and write $[d] = \{1, 2, \dots, d\}$. We view an element of $\{0,1\}^{d}$, the Euclidean cube, either as a sequence $\epsilon = (\epsilon_{1}, \dots, \epsilon_{d})$ of 0's and 1's; or as a subset of $[d]$. A subset $\epsilon$ corresponds to the sequence $(\epsilon_{1}, \dots, \epsilon_{d}) \in \{0,1\}^{d}$ such that $i \in \epsilon$ if and only if $\epsilon_{i} = 1$ for $i \in [d]$. For example, $\overrightarrow{0} = (0, \dots, 0) \in \{0,1\}^{d}$ is the same as $\emptyset \subset [d]$ and $\overrightarrow{1} = (1, \dots, 1)$ is the same as $[d]$.

If $X$ is a set, we denote $X^{2^{d}}$ by $X^{[d]}$ and we write a point $\bx \in X^{[d]}$ as $\bx = (x_{\epsilon} : \epsilon \subseteq [d])$. For example, for $d=2$ we have $\bx=(x_{\emptyset},x_{\{1\}},x_{\{2\}},x_{\{1,2\}})$. Sometimes, it is convenient to view an element of $X^{[d]}$ as a map from $\{0,1\}^{d}$ to $X$. We can isolate the first coordinate, writing $X^{[d]}_{*} = X^{2^{d-1}}$ and writing a point $\bx\in X^{[d]}$ as $\bx = (x_{\emptyset},\bx_{*})$, where $\bx_{*} = (x_{\epsilon}: \epsilon\neq \emptyset)\in X^{[d]}_{*}$. For a point $x\in X$ we let $x^{[d]}\in X^{[d]}$ and $x^{[d]}_{*}\in X^{[d]}_{*}$ be the diagonal points all of whose coordinates are $x$.

Let $0\leq \ell\leq d$ be an integer. An {\em $\ell$-dimensional face} of $\{0,1\}^{d}$, or equivalently a {\em face of codimension $d-\ell$} of $\{0,1\}^{d}$, is a subset of $\{0,1\}^{d}$ obtained by fixing the values of $d-\ell$ coordinates. We write dim($\alpha$) and codim($\alpha$) for the dimension and codimension of a face $\alpha$. A {\em facet} of $\{0,1\}^{d}$ is defined to be a face of codimension $1$. A face $\alpha$ is an {\em upper face} if $\overrightarrow{1}\in\alpha$.

Let $L$ be a group and $d\geq 1$ be an integer. If $\alpha$ is a face of $\{0,1\}^{d}$, for $g\in L$ we define $g^{(\alpha)}\in L^{[d]}$ by\begin{align*}
    g^{(\alpha)}(\epsilon) = \left\lbrace \begin{matrix}
        g & \text{if } \epsilon\in\alpha\\
        e & \text{otherwise}.
    \end{matrix}\right.
\end{align*}

where $e$ is the identity element of $L$.

We call the subgroup of $L^{[d]}$ generated by all $g^{(\alpha)}$, where $g\in L$ and $\alpha$ is a facet of $\{0,1\}^{d}$, the {\em Host-Kra cube group} (of order $d$) associated with $L$ and denote it by $\mathcal{HK}^{[d]}(L)$. We call the subgroup of $L^{[d]}$ generated by all $g^{(\alpha)}$, where $g\in L$ and $\alpha$ is an upper facet of $\{0,1\}^{d}$, the {\em face cube group} and denote it by $\mathcal{F}^{[d]}(L)$. When there is no ambiguity, we omit the group $L$ from the notation.

The Host-Kra and face cube groups originally appeared in \cite[Section 5]{Host_Kra_nonconventional_averages_nilmanifolds:2005} and coincide with the parallelepiped groups and face groups respectively of {\cite[Definition 3.1]{Host_Kra_Maass_nilstructure:2010}} introduced for abelian actions. See also \cite[Appendix E]{Green_Tao_linear_eq_primes:2010} for treatment of Host-Kra cube groups in nilpotent Lie groups.

\subsection{Nilsystems and Host-Kra seminorms}

We refer to {\cite[Chapters 8, 10, 11 and 16]{Host_Kra_nilpotent_structures_ergodic_theory:2018}} for the material discussed in this section, see also \cite{Leibman15,Leibman05a}.

Let $L$ be a $d$-step nilpotent Lie group and $\Gamma$ a discrete cocompact subgroup of $L$. The compact manifold $X = L/\Gamma$ is called a {\em $d$-step nilmanifold}. Observe that $L$ acts naturally on $X$ by left translation. If $T$ is a topological group and $\phi: T\to L$ is a continuous homomorphism, the induced action $(X, T)$ is called a {\em nilsystem of order $d$}.

Let $(X,\mu,T)$ be a measure preserving system. We define, by induction, a probability measure $\mu^{[d]}$ on $X^{[d]}$ that is invariant under $T$. Set $\mu^{[0]}=\mu$. For $d\geq 0$, let $\mathcal{I}^{[d]}$ be the $\sigma$-algebra of $T$-invariant subsets of $X^{[d]}$. We define $\mu^{[d+1]}$ to be the relatively independent square of $\mu^{[d]}$ over $\mathcal{I}^{[d]}$.

For a bounded function $f$ on $X$, define the seminorm\begin{align*}
    \nnorm{f}_{d} \coloneqq \left( \int_{X^{[d]}}\prod_{j=0}^{2^{d}-1}f(x_{j})\, d\mu^{[d]}(\bx) \right)^{1/2^{d}}.
\end{align*}

These seminorms define factors of $\cX$. More precisely, the sub-$\sigma$-algebra $\cZ_{d-1}(X)$ of $\cX$ is characterized by\begin{align*}
    \E(f\mid \cZ_{d-1}(X))=0\, \text{ if and only if } \nnorm{f}_{d}=0.
\end{align*}

We denote by $Z_{d}(X)$ the factor of $X$ associated to $\cZ_{d}(X)$. When there is no ambiguity, we write $Z_{d}$ and $\cZ_{d}$ instead of $Z_{d}(X)$ and $\cZ_{d}(X)$.

In \cite{Host_Kra_nonconventional_averages_nilmanifolds:2005}, a structure theorem for systems satisfying $Z_{d}(X)=X$  was proved for $\Z$-actions. Later, in \cite{Griesmer_thesis_averages_correlation_sumsets:2009}, this result was extended to $\Z^{d}$-actions. Recently, it was extended to actions of finitely generated nilpotent groups in \cite{Candela_Balazs_nilspace_general_seminorm_exchange_limits:2023}.

 \begin{theorem}[{\cite[Theorem 1.2]{Candela_Balazs_nilspace_general_seminorm_exchange_limits:2023}}]\label{thm: str_host_kra_factors}
     Let $(X,\cX,\mu,T)$ be an ergodic measure preserving system with $T$ being a finitely generated nilpotent group. Then, for each positive integer $d$, $Z_{d}$ is isomorphic to an inverse limit of $d$-step nilsystems.
 \end{theorem}

\subsection{Dynamical cubes} Let $(X,T)$ be a system. The Host-Kra cube group of $T$ acts on $X^{[d]}$ by\begin{align*}
    (t\bx)(\epsilon)=t_{\epsilon}\bx_{\epsilon}
\end{align*}

for $t\in \mathcal{HK}^{[d]}(T), \bx\in X^{[d]}$ and $\epsilon\subseteq [d]$.

Let $d\geq 1$ be an integer. Following Host, Kra and Maass {\cite[Definition 1.1]{Host_Kra_Maass_nilstructure:2010}} the set of {\em dynamical cubes of dimension $d$} is \begin{align*}
    \bQ^{[d]}(X)=\overline{\{tx^{[d]}:t\in\mathcal{HK}^{[d]}(T),x\in X\}}.
\end{align*}

Note that if $(X,T)$ is minimal \begin{align*}
        \bQ^{[d]}(X)= \overline{\{tx_{0}^{[d]}:t\in\mathcal{HK}^{[d]}(T)\}} = \overline{\{tx^{[d]}:t\in\mathcal{F}^{[d]}(T),x\in X\}}.
    \end{align*}

for each $x_{0}\in X$.

\begin{theorem}[{\cite[Propositions 4.6 and 4.8, Theorems 4.10 and 4.16]{Glasner_Gutman_Ye_higher_regionallyproximal_general_groups:2018}}]\label{thm: cube_is_minimal}
    Let $(X, T)$ be a minimal system and $d\geq 1$ be an integer. Let\begin{align*}
        \bQ_{x}^{[d]}(X) &= \bQ^{[d]}(X) \cap (\{x\}\times X^{2^{d}-1})\\
        Y_{x}^{[d]} &= \overline{\mathcal{F}^{[d]}(T)x^{[d]}},
    \end{align*}

    where $x\in X$. Then,\begin{enumerate}[label=(\arabic*)]
        \item $(\bQ^{[d]}(X),\mathcal{HK}^{[d]}(T))$ is a minimal topological dynamical system.
        \item For each $x\in X$, $(Y_{x}^{[d]},\mathcal{F}^{[d]}(T))$ is a minimal topological dynamical system.
        \item For each $x\in X$, $Y_{x}^{[d]}$ is the unique minimal subsystem of $(\bQ_{x}^{[d]}(X),\mathcal{F}^{[d]}(T))$.
        \item $\cF^{[d]}$-minimal points are dense in $\bQ^{[d]}(X)$.
        \item If $(X,T)$ is a metric system, then there exists a dense $G_{\delta}$ subset $X_{0}\subseteq X$ such that $Y_{x}^{[d]} = \bQ_{x}^{[d]}(X)$ for every $x\in X_{0}$.
    \end{enumerate}
\end{theorem}
\subsection{Higher order regionally proximal relations}

\subsubsection{Regionally proximal relation of order $d$}

Let $(X, T)$ be a system and $d\geq 1$ be an integer. A pair $(x,y)\in X\times X$ is said to be {\em regionally proximal of order $d$} if there are nets $f_{i}\in \cF^{[d]}$, $x_{i},y_{i}\in X$, and $a_{*}\in X^{[d]}_{*}$ such that $(f_{i}x^{[d]}_{i},f_{i}y_{i}^{[d]})\to (x,a_{*},y,a_{*})$. The set of regionally proximal pairs of order $d$ is denoted by $\RP^{[d]}(X,T)$ and is called the {\em regionally proximal relation of order $d$}. When there is no ambiguity, we write $\RP^{[d]}(X)$ instead of $\RP^{[d]}(X,T)$. Note that $\RP^{[1]}(X)$ is the classical regionally proximal relation.

The relation $\RP^{[d]}(X)$ is a closed and invariant relation. Moreover, (see \cite[Lemma A.5]{Glasner_Gutman_Ye_higher_regionallyproximal_general_groups:2018})\begin{align*}
    \P(X)\subseteq \cdots \subseteq \RP^{[d+1]}(X)\subseteq \RP^{[d]}(X)\subseteq \cdots \subseteq \RP^{[1]}(X). 
\end{align*}

\subsubsection{Nilpotent regionally proximal relation of order $d$}

Let $(X, T)$ be a system and $d\geq 1$ be an integer. A pair $(x,y)\in X\times X$ is said to be {\em nilpotent regionally proximal of order $d$} if $(x,y^{[d+1]}_{*})\in \bQ^{[d+1]}(X)$. The set of nilpotent regionally proximal pairs of order $d$ is denoted by $\NRP^{[d]}(X,T)$ and is called the {\em nilpotent regionally proximal relation of order $d$}. When there is no ambiguity, we write $\NRP^{[d]}(X)$ instead of $\NRP^{[d]}(X,T)$. We say that $(X, T)$ is a {\em system of order $d$} if $\NRP^{[d]}(X)$ coincides with the diagonal relation.

The relation $\NRP^{[d]}(X)$ is a closed and invariant relation. Moreover, (see {\cite[Lemma A.5]{Glasner_Gutman_Ye_higher_regionallyproximal_general_groups:2018}})

\begin{align*}
    \P(X)\subseteq \cdots \subseteq \NRP^{[d+1]}(X)\subseteq \NRP^{[d]}(X)\subseteq \cdots \subseteq \NRP^{[1]}(X). 
\end{align*}

The following theorem shows some properties of the nilpotent regionally proximal relation of order $d$.

\begin{theorem}[{\cite[Corollary 4.2, Theorems 3.8 and 6.1, and Proposition 8.9]{Glasner_Gutman_Ye_higher_regionallyproximal_general_groups:2018}}]\label{thm: NRP_properties}
    Let $\pi:(X,T)\to (Y,T)$ be a factor between minimal systems and let $d\geq 1$ be an integer. Then:\begin{enumerate}[label=(\arabic*)]
        \item $(x,y)\in \NRP^{[d]}(X)$ if and only if $(x,y^{[d+1]}_{*}) \in \overline{\cF^{[d+1]}x^{[d+1]}}$.
        \item $\NRP^{[d]}(X)$ is an equivalence relation.
        \item $\pi\times \pi (\NRP^{[d]}(X)) = \NRP^{[d]}(Y)$.
        \item $\RP^{[d]}(X)\subseteq \NRP^{[d]}(X)$.
    \end{enumerate}
\end{theorem}

It is worth mentioning that $\NRP^{[d]}$ is not necessarily equal to $\RP^{[d]}$ (see {\cite[Remark 8.10]{Glasner_Gutman_Ye_higher_regionallyproximal_general_groups:2018}}). Nevertheless, for abelian actions, these two relations coincide ({\cite[Theorem 3.2]{Shao_Ye_regionally_prox_orderd:2012}}).

Under certain conditions, systems of order $d$ have a structure theorem, analogous to the ergodic structure theorem.

\begin{theorem}[{\cite[Theorem 1.4]{Gutman_Manners_Varju_nilspaces_III:2020}}]\label{thm: structure_thm_NRP}
    Let $(X,T)$ be a minimal metric system and let $d\geq 1$ be an integer. Suppose that $T$ has a dense subgroup generated by a compact set. If $(X,T)$ is a system of order $d$ then $(X,T)$ is isomorphic to an inverse limit of $d$-step nilsystems.
\end{theorem}

The assumption that the system is metric in \cref{thm: structure_thm_NRP} can be removed for countable actions, by the following result.

\begin{proposition}[{\cite[Proposition 4.1]{Glasner_Hereditarily_sensitive:2006}}]\label{prop: inverse_limit_metric}
    Let $(X,T)$ be a system. If $T$ is countable, then $(X,T)$ is an inverse limit of metric systems.
\end{proposition}

\subsection{Enveloping semigroups}

The {\em enveloping semigroup} (or {\em Ellis semigroup}) $E(X, T)$ of a system $(X, T)$ is defined as the closure of $T$ in $X^X$ endowed with the product topology. For an enveloping semigroup $E(X, T)$, the maps $E(X, T) \to E(X, T)$, $p \mapsto pq$ and $p \mapsto tp$ are continuous for all $q \in E(X,T)$ and $t\in T$.

Ellis introduced this notion, and it has proved to be a useful tool in studying dynamical systems. The algebraic properties of $E(X)$ can be converted into the dynamical properties of $(X, T)$ and the reverse is also true. To exemplify this, we recall the following theorem.

\begin{theorem}[see {\cite[Chapters 3,4 and 5]{Auslander_minimal_flows_and_extensions:1988}}]\label{thm: distal_equi_chr_enveloping_semigroup}
    Let $(X, T)$ be a system. Then, \begin{enumerate}[label=(\arabic*)]
        \item $(X, T)$ is distal if and only if $E(X, T)$ is a group.
        \item $(X, T)$ is equicontinuous if and only if $E(X, T)$ is a group of continuous transformations.
        \item If $T$ is an abelian group, $(X,T)$ is equicontinuous if and only if $E(X,T)$ is an abelian group.
    \end{enumerate}
\end{theorem}

\subsection{The universal system} \label{sec:univ_system}
We refer to \cite[Sections IV, V and VI]{deVries_elements_topological_dynamics:1993} for the material discussed in this section. It is known that the semigroup $\beta T$, the Stone–Čech compactification of the discrete group $T$, is the universal point transitive system.  That is, for every transitive system $(X, T)$ and a point $x\in X$ with dense orbit, there exists an extension of systems $(\beta T, T)\to (X, T)$ which sends $e$, identity element of $T$, onto $x$. Therefore, by universality, there exists a unique extension $\Phi_{X}:(\beta T, T)\to (E(X, T), T)$, which is also a semigroup homomorphism, and we can interpret the $\beta T$ action on $X$ via this homomorphism.

The semigroup $\beta T$ admits many minimal left ideals, which coincide with the minimal subsystems. All these ideals are isomorphic to each other both as compact right topological semigroups and as minimal systems. We will fix a minimal left ideal $M$ of $\beta T$.  The universality of $\beta T$ implies that $(M, T)$ is the universal minimal system. For a semigroup, the element $v$ with $v^{2}=v$ is called an {\em idempotent}. By the Ellis Namakura theorem, the set $J(M)$ of idempotents in $M$ is nonempty. Moreover, $vM$ is a group with identity $v$, where $v\in J(M)$. We fix one such idempotent $u\in J(M)$.

If $(X,T)$ is minimal, then $X=Mx$ for every $x\in X$. A necessary and sufficient condition for $x$ to be minimal is that $ux=x$ for some $u\in J(M)$. A minimal system $(X, T)$ is distal if and only if $X=vX$ for every $v\in J(M)$. So, a minimal system $(X,T)$ is distal iff $\Phi_{X}(M)=\Phi_{X}(vM)=E(X,T)$ for every $v\in J(M)$. Also, we can characterize the proximal points: a pair $(x,y)\in X\times X$ is proximal if and only if there exists a minimal left ideal $I$ of $\beta T$ such that $y=vx$ for some $v\in J(I)$. 

Let $2^{X}$ be the collection of nonempty closed subsets of $X$ endowed with the Hausdorff topology. The action of $T$ on $2^{X}$ is given by\begin{align*}
    tA=\{ta:a\in A\}
\end{align*}

for each $t\in T$ and $A\in 2^{X}$. This action induces another action of $\beta T$ on $2^{X}$, and we denote this action by the circle operation: $p\circ_{T} A$, where $p\in \beta T$ and $A\subset X$, to distinguish it from the subset $pA$. For $p\in \beta T$ and $A\subseteq X$, we define $p\circ_{T} A=p\circ_{T} \overline{A}$ and $p\circ_{T} \emptyset=\emptyset$, where $\overline{A}$ denotes the closure of $A$ in the usual topology of $X$. It holds that \begin{align*}
    p\circ_{T} A = \{x\in X: \exists (t_{\lambda})_{\lambda\in\Lambda}\subseteq T, \exists (x_{\lambda})_{\lambda\in\Lambda}\subseteq A\text{ with } t_{\lambda}\to p, t_{\lambda}x_{\lambda}\to x\}.
\end{align*}

for all $p\in \beta T$ and $A\in 2^{X}$,  where $t_{\lambda}\to p$ denotes the convergence of $(t_{\lambda})_{\lambda\in \Lambda}$ to $p$ in the usual topology of $\beta T$. We use this notation for convergence in the usual topology throughout the paper. The circle operation has the following properties.

\begin{proposition}\label{prop: circle_arithmetic}
    Let $p\in M$ and $v,w\in J(M)$. Then\begin{enumerate}[label=(\arabic*)]
        \item $vp\circ_{T} wA=v\circ_{T} pA$.
        \item $p\circ_{T} wA = p\circ_{T} vA$.
        \item $p(w\circ_{T} A)=p(v\circ_{T} A)$.
        \item $(w\circ_{T} A)\cap wA = w(w\circ_{T} A)$.
    \end{enumerate}
\end{proposition}

The circle operation also behaves well with factor maps.

\begin{proposition}\label{prop: factors_circle}
    Let $\pi: (X,T)\to (Y,T)$ be a factor between systems, $p,q\in M$, $v,w\in J(M)$, $A\subseteq X$ and $y\in Y$. Then\begin{enumerate}[label=(\arabic*)]
        \item $\pi(p\circ_{T} A) = p\circ_{T} \pi(A)$.
        \item $p\circ_{T} \pi^{-1}(y) \subseteq \pi^{-1}(py)$ with equality if $\pi$ is open.
        \item $p\circ_{T} v\pi^{-1}(qy)= pq\circ_{T} v\pi^{-1}(wy)$.
    \end{enumerate}
\end{proposition}

The set $G=uM$ is a subgroup of $M$ with identity $u$. The set $\{vM:v\in J(M)\}$ is a partition of $M$ and each $p\in M$ has a unique representation $p=vg$, where $v\in J(M)$ and $g\in G$. We sometimes write $p^{-1}$ for $vg^{-1}$. The group $G$ plays a central role in the algebraic theory of minimal systems. It carries a $T_{1}$ compact topology, called by Ellis the {\em $\uptau$-topology}, which is weaker than the relative topology induced on $G$ as a subset of $M$. This topology was first introduced by Furstenberg in \cite{Furstenberg_structure_distal_flows:1963}, and developed by Ellis, Glasner and Shapiro in \cite{Ellis_Glasner_Shapiro_PI-flows:1975}. For any subset $A\subseteq G$, the $\uptau$-topology is determined by\begin{align*}
    \cltau{A} = u(u\circ_{T} A) = (u\circ_{T} A) \cap G.
\end{align*}

Let $(g_{i})_{i\in I}$ be a net in $G$ and $g\in G$. The convergence of $(g_{i})_{i\in I}$ to $g$ with respect to the $\uptau$-topology is denoted by $g_{i}\taulim g$. 

For a $\uptau$-closed subgroup $F$ of $G$ the {\em derived group} $H(F)$ is given by:\begin{align*}
    H(F)=\bigcap\{\cltau{V}:V\text{ a } \uptau\text{-open neighborhood of } u \text{ in } F\}.
\end{align*}

The group $H(F)$ is a $\uptau$-closed normal subgroup and is the smallest $\uptau$-closed subgroup $H$ of $F$ such that $F/H$ is a compact Hausdorff topological group (for the quotient topology induced by the $\uptau$-topology).

For a point $x_{0}\in uX$, the {\em Ellis group} of the pointed system $(X,x_{0})$ is the $\uptau$-closed subgroup \begin{align*}
    \GG(X,x_{0})=\{g\in G:gx_{0}=x_{0}\}.
\end{align*}

When there is no ambiguity, we omit the fixed point $x_{0}$ and denote the Ellis group of $X$ by $\GG(X)$. Whenever $x_0$ is used, we will refer to the element with respect to which the group $\GG(X)$ is defined. For $\pi:(X,T)\to (Y,T)$ a factor of minimal systems we have $\GG(X)\subseteq \GG(Y)$, where the group $\GG(Y)$ is defined by the point $\pi(x_{0})$. The Ellis group is very useful when studying factors. Below, there are a few applications of this concept. 

\begin{theorem}[see {\cite[Chapters 10 and 14]{Auslander_minimal_flows_and_extensions:1988}}]\label{thm: extensions_ellis_group}
    Let $\pi:(X, T) \to (Y, T)$ be a factor between minimal systems. Then,\begin{enumerate}[label=(\arabic*)]
        \item $\pi$ is proximal if and only if $\GG(X)= \GG(Y)$.
        \item $\pi$ is distal if and only if $\pi^{-1}(py_{0})=p\GG(Y)x_{0}$, for all $p\in M$.
        \item $\pi$ is equicontinuous if and only if $\pi$ is distal and $H(\GG(Y))\subseteq \GG(X)$.
        \item $\pi$ is a group factor if and only if $\pi$ is distal and $H(\GG(Y))\unlhd \GG(X)$.
    \end{enumerate} 
\end{theorem}

Let $\pi:(X, T) \to (Y, T)$ be a factor of minimal systems. We say that $\pi$ is {\em RIC (relatively incontractible)} if for every $p\in M$\begin{align*}
    \pi^{-1}(py_{0}) = p\circ_{T} u\pi^{-1}(y_{0})=p \circ_{T} Fx_{0},
\end{align*}

where $F=\GG(Y)$ and $y_{0}=\pi(x_{0})$.

\subsection{The AG diagram}

The following is a classical theorem stating that every factor map between minimal metric systems can be lifted to an open factor by almost one-to-one modifications.

\begin{theorem}[see {\cite[Chapter VI]{deVries_elements_topological_dynamics:1993}}]\label{thm: AG_diagram}
    Given $\pi: (X, T) \to (Y, T)$ a factor between minimal metric systems, there exists a commutative diagram of factors (called AG-diagram)
    \[\begin{tikzcd}
	X && {X'} \\
	\\
	{Y} && {Y'}
	\arrow["\pi"', from=1-1, to=3-1]
	\arrow["{\theta'}"', from=1-3, to=1-1]
	\arrow["{\pi'}", from=1-3, to=3-3]
	\arrow["\theta", from=3-3, to=3-1]
\end{tikzcd}\]

    such that\begin{enumerate}[label=(\arabic*)]
        \item $(X',T)$ and $(Y',T)$ are metric systems.
        \item $\theta$ and $\theta'$ are almost one-to-one factors,
        \item $\pi'$ is an open factor,
        \item $X'$ is the unique minimal set in $R_{\pi\theta} = \{(x,y)\in X\times Y': \pi(x)=\theta(y)\}$ and $\theta'$ and $\pi'$ are the restriction to $X'$ of the projections of $X\times Y'$ onto $X$ and $Y'$ respectively.
    \end{enumerate}
\end{theorem}

\section{The $\uptau^{[d]}$-topology and $\RP^{[d]}$}

\subsection{The $\uptau^{[d]}$-topology}\label{sec: uptau_d_topology}

In this section, we define the $\uptau^{[d]}$-topology using the circle operation and study its properties. This topology will be crucial for our study of the regionally proximal relation of order $d$ for abelian actions. 

To define the $\uptau^{[d]}$-topology, we first need the following lemma:\begin{lemma}[{\cite[Lemma 4.7]{Glasner_Gutman_Ye_higher_regionallyproximal_general_groups:2018}}]
    Let $v\in J(M)$. Then $v^{[d]}_{*}$ is a minimal idempotent of $\beta \cF^{[d]}_{*}$.
\end{lemma}

We define the $\uptau^{[d]}$-topology in a system $(X,T)$, for $d\geq 1$, by introducing a closure operator. Recall that throughout this paper, $u$ denotes the idempotent of $M$ fixed in \cref{sec:univ_system} to define the group $G$.  Given a subset $A \subseteq uX$, we define $\cltauF{d}{A}$ as the set of all $x \in uX$ such that $x^{[d]}_{*} \in u^{[d]}_{*}(u^{[d]}_{*} \circ_{\cF^{[d]}_{*}} \Delta^{[d]}_{*}(A))$. Note that, by the definition of the Vietoris topology, $x\in\cltauF{d}{A}$ if and only if $x\in uX$ and, for any net $(f_{i})_{i}$ in $\cF^{[d]}_{*}$ with $f_{i} \to u^{[d]}_{*}$, there exists a net $(a_{i})_{i}$ in $A$ such that $f_{i}a_{i*}^{[d]} \to x^{[d]}_{*}$. Note that, by definition of the $\uptau^{[d]}$-topology we have the following inclusions\begin{align*}
    \uptau = \uptau^{[1]} \subseteq \uptau^{[2]} \subseteq \cdots \subseteq \uptau^{[d]} \subseteq \uptau^{[d+1]} \subseteq \cdots.
\end{align*}

The following result shows that the $\uptau^{[d]}$-topology is indeed a topology.

\begin{proposition}
    $\cltauF{d}{\cdot}$ is a closure operator.
\end{proposition}

\begin{proof}
    It is clear that $A\subseteq \cltauF{d}{A}$ and $\cltauF{d}{A} \subseteq \cltauF{d}{F}$ for any $A,F\subseteq uX$ with $A\subseteq F$. By definition of $\cltauF{d}{\cdot}$, we have $\Delta^{[d]}_{*}(\cltauF{d}{A})\subseteq u^{[d]}_{*} \circ_{\cF^{[d]}_{*}} \Delta^{[d]}_{*}(A)$. Therefore, $u^{[d]}_{*}(u^{[d]}_{*}\circ_{\cF^{[d]}_{*}} \Delta^{[d]}_{*}(\cltauF{d}{A})) \subseteq u^{[d]}_{*}(u^{[d]}_{*}\circ_{\cF^{[d]}_{*}} u^{[d]}_{*} \circ_{\cF^{[d]}_{*}} \Delta^{[d]}_{*}(A)) = u^{[d]}_{*}(u^{[d]}_{*} \circ_{\cF^{[d]}_{*}} \Delta^{[d]}_{*}(A))$. Thus, we obtain that $\cltauF{d}{\cltauF{d}{A}} = \cltauF{d}{A}$.
\end{proof}

\begin{theorem}
    The $\uptau^{[d]}$-topology on $uX$ is compact and $T1$.
\end{theorem}

\begin{proof}
    Note that $u\overline{A}\subseteq \cltauF{d}{A}$ for all $A\subseteq uX$. Let $(A_{i})_{i}$ be a family of $\uptau^{[d]}$-closed subsets of $uX$ with the finite intersection property. Therefore, $(\overline{A_{i}})_{i}$ has the finite intersection property in $X$. Since $X$ is a compact set, there exists $x\in \bigcap_{i} \overline{A_{i}}$. It follows from $u\overline{A}\subseteq \cltauF{d}{A}$, for any $A\subseteq uX$, that $ux \in \bigcap_{i}\cltauF{d}{A_{i}} = \bigcap_{i} A_{i}$. Hence, the $\uptau^{[d]}$-topology is compact. Note that, by definition of the circle operation, $\cltauF{d}{\{x\}} = \{x\}$ for all $x\in uX$, and consequently it is $T1$.
\end{proof}

Note that the definition of the $\uptau^{[d]}$-topology depends on the choice of the fixed idempotent $u$. The following result shows that this choice does not affect the properties of the $\uptau^{[d]}$-topology.

\begin{lemma}
    The map $x\mapsto vx$ defines a $\uptau^{[d]}$-isomorphism from $uX$ onto $vX$, where $v\in J(M)$.
\end{lemma}

\begin{proof}
    Let $A$ be a $\uptau^{[d]}$-closed subset of $uX$. By \cref{prop: circle_arithmetic}, we have
    \begin{align*}
        v^{[d]}_{*}(v^{[d]}_{*}\circ_{\cF^{[d]}_{*}} \Delta^{[d]}_{*}(A)) &= v^{[d]}_{*}(u^{[d]}_{*}\circ_{\cF^{[d]}_{*}} \Delta^{[d]}_{*}(A))\\
        &= v^{[d]}_{*} u^{[d]}_{*} (u^{[d]}_{*} \circ_{\cF^{[d]}_{*}} \Delta^{[d]}_{*}(A)).
    \end{align*} 

    Therefore, we obtain that\begin{align*}
        \cltauF{d}{vA} = v\cltauF{d}{A} = vA.
    \end{align*} 
    
    Thus, $vA$ is $\uptau^{[d]}$-closed in $vX$, and hence we conclude that the map is a $\uptau^{[d]}$-isomorphism of $uX$ onto $vX$.
\end{proof}

Let $(g_{i})_{i\in I}$ be a net in $G$ and $g\in G$. The convergence of $(g_{i})_{i\in I}$ to $g$ with respect to the $\uptau^{[d]}$-topology is denoted by $g_{i}\taulimF{d} g$.

\begin{lemma}\label{lemma: xi_to_x_xi_taulim_x}
    Let $(x_{i})_{i}$ be a net in $uX$ with $x_{i}\to x$, for some $x\in X$. Then $x_{i} \taulimF{d} ux$.
\end{lemma}

\begin{proof}
    Let $A_{j}= \{x_{i}: i\geq j\}$. Since $x_{i}\to x$, we have that $x\in \overline{A_{j}}$ for every $j$, and consequently $ux \in u\overline{A_{j}}$. It follows from $u\overline{A_{j}}\subseteq \cltauF{d}{A_{j}}$ that $ux\in \bigcap_{j}\cltauF{d}{A_{j}}$. Since we can apply the same reasoning to any subnet of $\{x_{i}\}_{i}$, we conclude that $x_{i}\taulimF{d} ux$.
\end{proof}

It is known that the inversion and left and right multiplication are $\uptau^{[1]}$-continuous. However, this continuity does not generally hold for higher orders. We will show instead that these operations are $\uptau^{[d]}$-continuous at points satisfying a suitable condition. We will first provide a basis for this topology. For this, we first recall some basic properties of the basis for the topology of the Stone-\v{C}ech compactification $\beta T$.

Since the group $T$ is discrete, for every subset $U\subseteq T$, the closure $\overline{U}$ (taken in $\beta T$) is both open and closed in $\beta T$, and $\overline{U}\cap T=U$. Moreover the family\begin{align*}
    \{\overline{U}: U\subseteq T\}
\end{align*}

form a basis for the topology of $\beta T$.

\begin{theorem}
    Let $(X,T)$ be a system and $x\in uX$. A basis for the $\uptau^{[d]}$-topology at $x$ is given by the sets\begin{align*}
        (U,V)_{d} = \bigcup_{f\in V} \bigcap_{\epsilon\neq\emptyset} f_{\epsilon}^{-1}U \cap uX
    \end{align*}

    where $V\subseteq \cF^{[d]}_{*}$ such that $u^{[d]}_{*}\in \overline{V}$ and $U$ is a neighborhood of $x$ in $X$.
\end{theorem}

\begin{proof}
    Let $W$ be a $\uptau^{[d]}$-open neighborhood of $x$ in $uX$, and set $C= uX\setminus W$, so $C$ is a $\uptau^{[d]}$-closed subset of $uX$. Since $C'=\phi(u^{[d]}_{*}\circ_{\cF^{[d]}_{*}} \Delta^{[d]}_{*}(C)\cap \Delta^{[d]}_{*}(C))$ is closed and $x\notin C'$, where $\phi$ is the projection onto the last coordinate, we have that $U = X\setminus C'$ is an open neighborhood of $x$ in $X$. Now, consider the family $(U_{i})_{i}$ of open neighborhoods of $x$ contained in $U$. For each $U_{i}$, define the open set $O_{i} = \{p\in M: px\in U_{i}\}$, and set $V_{i}\subseteq \cF^{[d]}_{*}$ such that $u^{[d]}_{*} \in \overline{V_{i}} \cap M^{[d]} \subseteq O_{i*}^{[d]} \cap \beta \cF^{[d]}_{*}$.

    Suppose that $(U_{i},V_{i})_{d}\cap C \neq\emptyset$ for all $i$, and consider $x_{i}\in (U_{i},V_{i})_{d}\cap C$. Then there exists $f_{i}\in V_{i}$ such that $f_{i}x_{i*}^{[d]} \in U_{i*}^{[d]}$. By the definitions of $U_{i}$ and $V_{i}$, it follows that $f_{i}x_{i*}^{[d]}\to x^{[d]}_{*}$ and $f_{i}x^{[d]}_{*}\to x^{[d]}_{*}$. Without loss of generality, by passing to a subnet if necessary, we may assume that $f_{i}\to f$. Therefore, we have $fx^{[d]}_{*}=x^{[d]}_{*}$, which in particular implies that $u^{[d]}_{*}f^{-1}x^{[d]}_{*} = x^{[d]}_{*}$. Thus, we obtain that \begin{align*}
        x^{[d]}_{*} = \lim f_{i}x_{i*}^{[d]} = \lim f_{i} u^{[d]}_{*} x_{i*}^{[d]}\in \lim f_{i}u^{[d]}_{*}\circ_{\cF^{[d]}_{*}} \Delta^{[d]}_{*}(C)= fu^{[d]}_{*} \circ \Delta^{[d]}_{*}(C),
    \end{align*}

    which implies that \begin{align*}
        x^{[d]}_{*} = u^{[d]}_{*}f^{-1}x^{[d]}_{*}= u^{[d]}_{*}f^{-1}u^{[d]}_{*}x^{[d]}_{*} \in u^{[d]}_{*} \circ_{\cF^{[d]}_{*}} \Delta^{[d]}(C).
    \end{align*}

    Therefore, we obtain that $x\in \cltauF{d}{C}=C$, which is a contradiction. Hence, there exists $i$ such that $(U_{i},V_{i})_{d}\cap C = \emptyset$, and we conclude that $(U_{i},V_{i})_{d}\subseteq W$.

    Conversely, let $(U,V)_{d}$ such that $u^{[d]}_{*}\in \overline{V}$ and $U$ is a neighborhood of $x$ in $X$, and define $C= uX\setminus (U,V)_{d}$. In particular, if $y\in C$ then, for every $f\in V$, we have that $fy^{[d]}_{*}\notin U^{[d]}_{*}$. Now, let $z\in \cltauF{d}{C}$. Then there exist a net $(f_{i})_{i}$ in $\cF^{[d]}_{*}$ and a net $(x_{i})_{i}$ in $C$ such that $z^{[d]}_{*} = \lim f_{i}x_{i*}^{[d]}$ and $f_{i}\to u^{[d]}_{*}$. Eventually, $f_{i}\in \overline{V}$ which is a neighborhood of $u^{[d]}_{*}$, and then $f_{i}\in V$. Therefore, we have that $f_{i}x_{i*}^{[d]} \notin U_{*}^{[d]}$, and consequently $z\notin U$. Thus, we deduce that\begin{align*}
        x\in U\cap uX \subseteq uX\setminus\cltauF{d}{C} = \mathrm{int}_{\uptau^{[d]}}((U,V)_{d}).
    \end{align*}

    Hence, we conclude that $(U,V)_{d}$ is a neighborhood of $x$.
\end{proof}

In particular, for $X=M$, using the fact that the family ${\overline{U}:U\subseteq T}$ is a basis for $\beta T$, we obtain the following basis for the $\uptau^{[d]}$-topology on $G$.

\begin{corollary}
    A basis for the $\uptau^{[d]}$-topology at $g\in G$ is given by the sets\begin{align*}
        (U,V)_{d} = \bigcup_{f\in V} \bigcap_{\epsilon\neq\emptyset} f_{\epsilon}^{-1}\overline{U} \cap G
    \end{align*}

    where $V\subseteq \cF^{[d]}_{*}$ such that $u^{[d]}_{*}\in \overline{V}$ and $U\subseteq T$ such that $g\in \overline{U}$.
\end{corollary}

First, we prove that the left multiplication is a $\uptau^{[d]}$-homeomorphism.

\begin{proposition}
    Left multiplication in $G$ is a $\uptau^{[d]}$-homeomorphism.
\end{proposition}

\begin{proof}
    Let $g\in G$ and let $A$ be a $\uptau^{[d]}$-closed subset of $G$. Let $h\in \cltauF{d}{Ag}$, so we have \begin{align*}
        h^{[d]}_{*} \in u^{[d]}_{*}(u^{[d]}_{*}\circ_{\cF^{[d]}_{*}} \Delta^{[d]}_{*}(Ag)) = u^{[d]}_{*}(u^{[d]}_{*}\circ_{\cF^{[d]}_{*}} \Delta^{[d]}_{*}(A))g^{[d]}_{*}.
    \end{align*}

    Therefore, we get that $hg^{-1} \in \cltauF{d}{A} = A$, and consequently $\cltauF{d}{Ag} = Ag$. That is, the left multiplication is a $\uptau^{[d]}$-homeomorphism.
\end{proof}

Now, we show that the right multiplication by an element under a certain condition is continuous with respect to the $\uptau^{[d]}$-topology.

\begin{proposition}\label{prop: right_mult_taud}
    Let $g\in G$ be such that $g^{[d]}_{*} \in \overline{\cF^{[d]}_{*}u^{[d]}_{*}}$. Then $h\mapsto gh$ is $\uptau^{[d]}$-continuous.
\end{proposition}

\begin{proof}
    Let $A$ be a $\uptau^{[d]}$-closed subset of $G$. Since $g^{[d]}_{*} \in \overline{\cF^{[d]}_{*}u^{[d]}_{*}}$, we have\begin{align*}
         u^{[d]}_{*}(u^{[d]}_{*}\circ_{\cF^{[d]}_{*}} g^{[d]}_{*}\Delta^{[d]}_{*}(A))&= u^{[d]}_{*}(g^{[d]}_{*}\circ_{\cF^{[d]}_{*}} \Delta^{[d]}_{*}(A))\\
                      &= g^{[d]}_{*}(g^{-1})^{[d]}_{*}(g^{[d]}_{*}\circ_{\cF^{[d]}_{*}} \Delta^{[d]}_{*}(A))\\
                      &= g^{[d]}_{*}((g^{-1})^{[d]}_{*} \circ (g^{[d]}_{*}\circ_{\cF^{[d]}_{*}}\Delta^{[d]}_{*}(A)))\\
                      &= g^{[d]}_{*}(u^{[d]}_{*}\circ_{\cF^{[d]}_{*}}\Delta^{[d]}_{*}(A)).
    \end{align*}

    Therefore, it follows that $\cltauF{d}{gA} = g\cltauF{d}{A} = gA$, completing the proof.
\end{proof}

Note that for $d=1$ every element of $G$ satisfies the condition of \cref{prop: right_mult_taud}, so we recover that the right multiplication is a $\uptau$-homeomorphism. Nevertheless, for $d>1$, we are unable to prove that the right multiplication is a $\uptau^{[d]}$-homeomorphism. We believe that \cref{prop: right_mult_taud} is the best we can expect.

Now, we can prove the inversion is $\uptau^{[d]}$-continuous at points that satisfy the same condition as in \cref{prop: right_mult_taud}. Before, we need the following lemma.

\begin{lemma}\label{lemma: condition_inversion_basis}
    For $U\subseteq T$, $V\subseteq \cF^{[d]}_{*}$ and $g\in G$. If $g^{[d]}_{*} \in \bigcup_{\textbf{g}\in U^{[d]}_{*}}\textbf{g}^{-1}\overline{V} \cap G^{[d]}_{*}$, then $g\in (U,V)_{d}^{-1}$.
\end{lemma}

\begin{proof}
    Let $g\in G$ such that $g^{[d]}_{*} \in \bigcup_{\bg\in U^{[d]}_{*}}\bg^{-1}\overline{V} \cap G^{[d]}_{*}$. Therefore, there is a $\bg \in U^{[d]}_{*}$ such that $\bg g^{[d]}_{*}\in U^{[d]}_{*} \in \overline{V}$. Thus, there exists a net $(\bv_{i})_{i\in I}$ in $V$ such that $\bv_{i}(g^{-1})^{[d]}_{*} \to \bg$. Since $U$ is an open set in $\beta T$, there exists $\bv \in \overline{V}$ such that $\bv (g^{-1})^{[d]}_{*} \in U^{[d]}_{*}$. Hence, we conclude that $g^{-1} \in (U,V)_{d}$.
\end{proof}

\begin{theorem}\label{thm: inversion_taud}
    Let $g\in G$ be such that $g^{[d]}_{*} \in \overline{\cF^{[d]}_{*}u^{[d]}_{*}}$. Then the map $h\mapsto h^{-1}$ of $G$ onto itself is $\uptau^{[d]}$-continuous at $g$.
\end{theorem}

\begin{proof}
    First we show that the inversion is continuous at $u$. Let $U\subseteq T$ and $V\subseteq \cF^{[d]}_{*}$ be such that $u\in \overline{U}$ and $u^{[d]}_{*}\in \overline{V}$, and define $C=G\setminus (U,V)_{d}^{-1}$. Let $h\in \cltauF{d}{C}$, in particular we have $h^{[d]}_{*} \in u^{[d]}_{*}\circ_{\cF^{[d]}_{*}} \Delta^{[d]}_{*}(C)$. Then there exist a net $(f_{i})_{i}$ in $\cF^{[d]}_{*}$ and a net $(g_{i})_{i}$ in $C$ such that $f_{i}g_{i*}^{[d]}\to h^{[d]}_{*}$ and $f_{i}\to u^{[d]}_{*}$. Therefore, eventually $f_{i}\in U^{[d]}_{*}$, and consequently, by \cref{lemma: condition_inversion_basis}, we have $f_{i}g_{i*}^{[d]} \notin \overline{V}$. Thus, we deduce that $h^{[d]}_{*}\notin \overline{V}$, and hence\begin{align*}
        u^{[d]}_{*} \in \overline{V}\cap \Delta^{[d]}_{*}(G) \subseteq \Delta^{[d]}_{*}(G \setminus \cltauF{d}{C}) = \Delta^{[d]}_{*}(\mathrm{int}_{\uptau^{[d]}}((U,V)_{d}^{-1})).
    \end{align*} 

    Thus, we obtain that $(U,V)_{d}^{-1}$ is a neighborhood of $u$, and hence we obtain that the inversion is $\uptau^{[d]}$-continuous at $u$. 

    Now, let $g\in G$ be such that $g^{[d]}_{*} \in \overline{\cF^{[d]}_{*}u^{[d]}_{*}}$ and let $(g_{i})_{i\in I}$ be a net in $G$ such that $g_{i} \taulimF{d} g$. By \cref{prop: right_mult_taud}, we have $g^{-1}g_{i}\taulimF{d} u$. Since the inversion map is $\uptau^{[d]}$-continuous at $u$, it follows that $g_{i}^{-1}g \taulimF{d} u$. Hence, we conclude that $g_{i}^{-1}\taulimF{d} g^{-1}$, completing the proof.
\end{proof}

The following theorem gives some topological properties of factor maps with respect to the $\uptau^{[d]}$-topology.

\begin{theorem}\label{thm: pi_u_props}
    Let $\pi:(X,T)\to(Y,T)$ be a factor between minimal systems, and set $\pi_{u}=\pi|_{uX}$. Then,\begin{enumerate}[label=(\roman*)]
        \item $\pi_{u}$ is $\uptau^{[d]}$-continuous.
        \item $\pi_{u}$ is a $\uptau^{[d]}$-homeomorphism if and only if $\pi$ is a proximal factor.
    \end{enumerate}
\end{theorem}

\begin{proof}
    \begin{enumerate}[label=(\roman*)]
        \item Let $A\subseteq uX$ and $x\in \cltauF{d}{A}$. In particular, we have $x^{[d]}_{*} \in u^{[d]}_{*}\circ_{\cF^{[d]}_{*}} \Delta^{[d]}_{*}(A)$. By \cref{prop: factors_circle}, it follows that\begin{align*}
            (\pi_{u}(x))^{[d]}_{*} \in u^{[d]}_{*}(u^{[d]}_{*} \circ_{\cF^{[d]}_{*}} \Delta^{[d]}_{*}(\pi_{u}(A))).
        \end{align*}

        Thus, we obtain that $\pi_{u}(x)\in \cltauF{d}{\pi_{u}(A)}$, and hence $\pi_{u}$ is $\uptau^{[d]}$-continuous.

        \item Note that $\pi$ is a proximal factor if and only if $\pi_{u}$ is bijective, so it suffices to prove the non-trivial implication. Suppose that $\pi$ is a proximal factor, it is enough to show that $\pi_{u}$ is closed.

        Let $A\subseteq uX$ and $y\in \cltauF{d}{\pi_{u}(A)}$. Therefore, there exist nets $(f_{i})_{i}$ in $\cF^{[d]}_{*}$ and $(a_{i})_{i}$ in $A$ such that $f_{i}\to u^{[d]}_{*}$ and $f_{i}\pi(a_{i})_{*}^{[d]}\to y^{[d]}_{*}$. Consider $x\in uX$ such that $y=\pi(x)$, and we may assume that there exists $\bold{a} \in u^{[d]}_{*}X^{[d]}_{*}$ such that $f_{i}a_{i*}^{[d]} \to\bold{a}$. Since $\pi(\bold{a})=y^{[d]}_{*}$ and $\pi$ is a proximal factor, it follows that $u^{[d]}_{*}\bold{a} = x^{[d]}_{*}$. Thus, we have that \begin{align*}
            x^{[d]}_{*} = u^{[d]}_{*}\bold{a} \in u^{[d]}_{*}(u^{[d]}_{*} \circ_{\cF^{[d]}_{*}} \Delta^{[d]}_{*}(A)).
        \end{align*}

        Hence, we obtain that $y \in \pi_{u}(\cltauF{d}{A})$, and we conclude that $\pi_{u}$ is $\uptau^{[d]}$-closed.
    \end{enumerate}
\end{proof}

Let $F$ be a $\uptau^{[d]}$-closed subgroup of $G$ with $\GG(X)\subseteq F$. Let $\cN_{\uptau^{[d]}}^{F}(X)$ denote the open neighborhood filter at $x_{0}$ of $Fx_{0}$ for the $\uptau^{[d]}$-topology. We define\begin{align*}
    H_{\uptau^{[d]}}(Fx_{0})=\bigcap\{\cltauF{d}{V}: V\in\cN_{\uptau^{[d]}}^{F}(X)\}.
\end{align*}

Note that $x\in H_{\uptau^{[d]}}(Fx_{0})$ if and only if there exists a net $(x_{i})_{i}$ in $Fx_{0}$ such that $x_{i}\taulimF{d} x$ and $x_{i}\taulimF{d} x_{0}$. This object will be essential in giving an algebraic characterization of the relation $\RP^{[d]}$.

\begin{lemma}\label{lemma: char_Hd}
 Let $F$ be a $\uptau^{[d]}$-closed subgroup of $G$ and let $g\in F$. Then $g\in H_{\uptau^{[d]}}(F)$ if and only if there exist $(f_{i})_{i}$ a net in $\cF^{[d]}$ and nets $(g_{i})_{i}, (g_{i}')_{i}$ in $F$ such that $f_{i}\to (e,u^{[d]}_{*})$ and $(f_{i}g_{i}^{[d]},f_{i}{g_{i}'}^{[d]})\to (r,g^{[d]}_{*},r,u^{[d]}_{*})$ for some $r\in M$. 
\end{lemma}

\begin{proof}
    Suppose that there exist $(f_{i})_{i}$ a net in $\cF^{[d]}$ and nets $(g_{i})_{i}, (g_{i}')_{i}$ in $F$ such that $f_{i}\to (e,u^{[d]}_{*})$ and $(f_{i}g_{i}^{[d]},f_{i}{g_{i}'}^{[d]})\to (r,g^{[d]}_{*},r,u^{[d]}_{*})$ for some $r\in M$. 
    By \cref{lemma: xi_to_x_xi_taulim_x}, it follows that $g_{i} \taulimF{d} ur$ and $g_{i}'\taulimF{d} ur$. Let $A_{j}=\{g_{i}:i\geq j\}$, since $f_{i}\to (e,u^{[d]}_{*})$, we obtain that $g^{[d]}_{*} \in u^{[d]}_{*} \circ_{\cF^{[d]}_{*}} \Delta^{[d]}_{*}(A_{i})$, and consequently $g_{i}\taulimF{d} g$. Similarly, we deduce that $g_{i}' \taulimF{d} u$. Hence, we have $ur\in H_{\uptau^{[d]}}(G)$, and we conclude that $g\in H_{\uptau^{[d]}}(F)$.

    Conversely, suppose that $g\in H_{\uptau^{[d]}}(F)$. Then there exists a net $(g_{i})_{i}$ in $F$ such that $g_{i}\taulimF{d} g$ and $g_{i} \taulimF{d} u$. Assume, without loss of generality, passing to a subnet if necessary, that $g_{i} \to r$ for some $r\in M$. Define $A_{i} = \{g_{j}:j\geq i\}$, so $g,u\in \cltauF{d}{A_{i}}$ for each $i$.
    
    Let $U$ be a neighborhood of $r$ in $M$. Then there exists $i_{U}$ such that $g_{i} \in U$ for all $i\geq i_{U}$. Let $W$ be a neighborhood of $u$ in $M$. By the definition of convergence in the Vietoris topology, there exist $i_{U,W}$, $f_{i_{W,U}}\in (f_{i})_{i}$ and $g_{i_{W,U}},g_{i_{W,U}}'\in A_{i_{U}}$ such that $f_{i_{W,U}}(\epsilon)g_{i_{W,U}}\in gW$ and $f_{i_{W,U}}(\epsilon)g_{i_{W,U}}'\in W$ for each $\epsilon\neq \emptyset$. Note that $(r,g^{[d]}_{*},r,u^{[d]}_{*})$ is a limit point of $(f_{i_{W,U}}g_{i_{W,U}}^{[d]},f_{i_{W,U}}{g_{i_{W,U}}'}^{[d]})$, and thus we obtain the desired result.
\end{proof}

A direct consequence of \cref{lemma: char_Hd} is the following.

\begin{proposition}\label{prop: H_taud_NRPd}
    Let $g\in H_{\uptau^{[d]}}(G)$. Then $(g,u)\in \RP^{[d]}(M)$, and consequently $g^{[d]}_{*} \in \overline{\cF^{[d]}_{*}u^{[d]}_{*}}$.
\end{proposition}

In particular, every $g\in H_{\uptau^{[d]}}(F)$ satisfies the hypothesis of \cref{prop: right_mult_taud} and \cref{thm: inversion_taud}. Using this, we obtain the following result.

\begin{theorem}\label{thm: Hd_prop}
    Let $F$ be a $\uptau^{[d]}$-closed subgroup of $G$. Then $H_{\uptau^{[d]}}(F)$ is a $\uptau^{[d]}$-closed subgroup of $F$.
\end{theorem}

\begin{proof}
        First, we show that $H_{\uptau^{[d]}}(F)V\subseteq \cltauF{d}{V}$ for each $V\in \cN_{\uptau^{[d]}}^{F}(M)$. Let $g\in H_{\uptau^{[d]}}(F)$ and $v\in V$. Since $g\in H_{\uptau^{[d]}}(F)$, there exists a net $(g_{i})_{i\in I}$ in $F$ such that $g_{i}\taulimF{d} g$ and $g_{i}\taulimF{d} u$. Therefore, we have $g_{i}v\taulimF{d} v$. Since $V$ is a $\uptau^{[d]}$-open set, it follows that $g_{i}v\in V$ eventually. Hence, it follows from $g_{i}v\taulimF{d} gv$ that $gv\in \cltauF{d}{V}$.

        Now, we prove that $H_{\uptau^{[d]}}(F)\cltauF{d}{V} \subseteq \cltauF{d}{V}$ for each $V\in \cN_{\uptau^{[d]}}^{F}(M)$. Let $g\in H_{\uptau^{[d]}}(F)$ and $v\in \cltauF{d}{V}$, and let $(v_{i})_{i\in I}$ be a net in $V$ such that $v_{i} \taulimF{d} v$. By \cref{prop: H_taud_NRPd} and \cref{prop: right_mult_taud}, we have $gv_{i} \taulimF{d} gv$. Since $(gv_{i})_{i\in I} \subseteq \cltauF{d}{V}$, we obtain that $gv \in \cltauF{d}{V}$. Thus, it follows that\begin{align*}
            H_{\uptau^{[d]}}(F)\bigcap_{V\in \cN_{\uptau^{[d]}}^{F}(M)}\cltauF{d}{V} \subseteq \bigcap_{V\in \cN_{\uptau^{[d]}}^{F}(M)} H_{\uptau^{[d]}}(F)\cltauF{d}{V} \subseteq \bigcap_{V\in \cN_{\uptau^{[d]}}^{F}(M)}\cltauF{d}{V}.
        \end{align*}

        Hence, we conclude that $H_{\uptau^{[d]}}(F)$ is a semigroup. Let $g\in H_{\uptau^{[d]}}(F)$, so $gH_{\uptau^{[d]}}(F)$ is also a semigroup, and by the Ellis Namakura theorem, it contains an idempotent. Since $gH_{\uptau^{[d]}}(F)\subseteq F$ and that $F$ is a group,  we have $u\in gH_{\uptau^{[d]}}(F)$. It follows that $g$ is invertible, and consequently $H_{\uptau^{[d]}}(F)$ is a group.
\end{proof}

Furthermore, the $\uptau^{[d]}$-topology also has the Ellis Trick.

\begin{lemma}\label{lemma: Ellis_trick}
    Let $F$ be a $\uptau^{[d]}$ closed subgroup of $G$ acting on $M$ by right multiplication, $(p,g)\mapsto pg: M\times F\to M$. Then\begin{enumerate}[label=(\roman*)]
        \item there is $w\in J(M) \cap \overline{F}$ such that $\overline{wF}$ is $F$-minimal.
        \item if $V$ is open subset of $\overline{wF}$, then $\mathrm{int}_{\uptau^{[d]}}\cltauF{d}{V\cap wF}\neq \emptyset$.
    \end{enumerate}
\end{lemma}

\begin{proof}
    \begin{enumerate}[label=(\roman*)]
        \item Clearly, $\overline{F}$ is an $F$-invariant closed subset of $M$, so there exists a minimal $F$-invariant subset $K \subseteq \overline{F}$. Let $p\in K$ and consider $w\in J(M)$ and $g\in G$ such that $p=wg$. Since $p \in \overline{F}$, there exists a net $(f_{i})_{i}$ in $F$ such that $f_{i} \to p$. Therefore, by \cref{lemma: xi_to_x_xi_taulim_x}, we have $f_{i} \taulimF{d} g$. Since $F$ is $\uptau^{[d]}$-closed, it follows that $g \in F$. Now, because $K$ is $F$-invariant, we obtain that $w=pg^{-1} \in KF \subseteq K$. Thus, by the minimality of $K$, we conclude that $\overline{wF}=K$.

        \item By the $F$-minimality of $\overline{wF}$ we have $\overline{wF} \subseteq VF$. Since $\overline{wF}$ is compact, it follows that there exist $f_{1},\dots,f_{n}$ in $F$ such that $\overline{wF} \subset \bigcup_{i=1}^{n}Vf_{i}$. In particular, we have $wF \subset \bigcup_{i=1}^{n}(V\cap wF)f_{i}$, and consequently $wF \subseteq \bigcup_{i=1}^{n}\cltauF{d}{V\cap wF}f_{i}$. Thus, at least one of the sets $\cltauF{d}{V\cap wF}f_{i}$ must have nonempty $\uptau^{[d]}$-interior. Hence, we conclude that $\cltauF{d}{V \cap wF}$ has nonempty $\uptau^{[d]}$-interior.
    \end{enumerate}
\end{proof}

We now introduce a new topology, inspired by the approach used by Furstenberg in his definition of the $\uptau$-topology. We will later show that this topology coincides with the $\uptau^{[d]}$-topology.

Let $d\geq 1$ be an integer. Let $\Sigma_{d}$ denote the set of all continuous pseudometrics on $X^{[d]}_{*}$. For every $\sigma\in \Sigma_{d}$ define $F_{\sigma}^{[d]}: X^{[d]}_{*}\times X^{[d]}_{*}\to \R$ and $F_{\sigma,d}:X\times X\to \R$ by\begin{align*}
    F_{\sigma}^{[d]}(\bx,\by) &= \inf\{\sigma(f\bx,f\by):f\in \cF^{[d]}_{*}\}\\
    F_{\sigma,d}(x,y) &= F_{\sigma}^{[d]}(x^{[d]}_{*},y^{[d]}_{*}).
\end{align*}

Note that $F_{\sigma}^{[d]}$ is symmetric, invariant under $\cF^{[d]}_{*}$ and upper semicontinuous, so $F_{\sigma,d}$ is symmetric and upper semicontinuous.

\begin{lemma}
    Let $(\bx,\by)\in X^{[d]}_{*}\times X^{[d]}_{*}$, $\sigma\in \Sigma_{d}$ and $\boldsymbol{p}\in \beta\cF^{[d]}_{*}$. Then $F_{\sigma}^{[d]}(\bx,\by)\leq F_{\sigma}^{[d]}(\boldsymbol{p}\bx,\boldsymbol{p}\by)$ with equality if $(\bx,\by)$ is a $\cF^{[d]}_{*}$-minimal point.
\end{lemma}

\begin{proof}
    Let $\varepsilon>0$. Since $F_{\sigma}^{[d]}$ is an upper semicontinuous function, the set\begin{align*}
        \{\boldsymbol{q}\in \beta\cF^{[d]}_{*}: F_{\sigma}^{[d]}(\boldsymbol{q}\bx,\boldsymbol{q}\by)<F_{\sigma}^{[d]}(\boldsymbol{p}\bx,\boldsymbol{p}\by) +\varepsilon\}
    \end{align*}

    is an open set. Therefore, since $\cF^{[d]}_{*}$ is dense in $\beta \cF^{[d]}_{*}$, there exists $f\in \cF^{[d]}_{*}$ such that \begin{align*}
        F_{\sigma}^{[d]}(\bx,\by) = F_{\sigma}^{[d]}(f\bx,f\by)<F_{\sigma}^{[d]}(\boldsymbol{p}\bx,\boldsymbol{p}\by) +\varepsilon.
    \end{align*}

    Thus, it follows that $F_{\sigma}^{[d]}(\bx,\by)\leq F_{\sigma}^{[d]}(\boldsymbol{p}\bx,\boldsymbol{p}\by)$.

    Now, if $(\bx,\by)$ is a $\cF^{[d]}_{*}$-minimal point. Then there exists $\boldsymbol{q}\in \beta\cF^{[d]}_{*}$ such that $(\boldsymbol{q}\boldsymbol{p}\bx,\boldsymbol{q}\boldsymbol{p}\by) = (\bx,\by)$, hence we conclude that $F_{\sigma}^{[d]}(\bx,\by)= F_{\sigma}^{[d]}(\boldsymbol{p}\bx,\boldsymbol{p}\by)$.
\end{proof}

Now define for every $x\in X$, $\sigma \in \Sigma_{d}$ and $\varepsilon>0$\begin{align*}
    U_{d}(x,\sigma,\varepsilon)=\{y\in X: F_{\sigma,d}(x,y)<\varepsilon\}.
\end{align*}

When there is no ambiguity, we write $U(x,\sigma,\varepsilon)$ instead of $U_{d}(x,\sigma,\varepsilon)$. We will show that these sets form a basis for a topology. To prove this, we first need the following lemmas.

\begin{lemma}[see {\cite[Chapter IX]{Bourbaki_topology_ch_5_10:1998}}]\label{lemma: uniform_pseudometrics}
    Given a uniformity $\mathcal{U}$ on a set $X$, there is a family of pseudometrics on $X$ such that the uniformity defined by this family is identical to $\mathcal{U}$.
\end{lemma}

\begin{lemma}
    Let $\sigma\in\Sigma_{d}$, $\varepsilon>0$, and let $(x,y)\in uX\times uX$ be such that $y\in U(x,\sigma,\varepsilon)$. Then there are $\rho\in \Sigma_{d}$ and $\delta>0$ such that $U(y,\rho,\delta)\subseteq U(x,\sigma,\varepsilon)$.
\end{lemma}

\begin{proof}
    Since $(x^{[d]}_{*},y^{[d]}_{*})$ is a $\cF^{[d]}_{*}$-minimal point, it follows that $Z=\overline{\cF^{[d]}_{*}(x^{[d]}_{*},y^{[d]}_{*})}$ is a $\cF^{[d]}_{*}$-minimal set. Let $c>0$ be such that $F_{\sigma,d}(x,y)<c<\varepsilon$ and let $V=\{\boldsymbol{z}\in Z: \sigma(\boldsymbol{z})<c\}$. Since $F_{\sigma,d}(x,y)<c$, we obtain that $V\neq \emptyset$. Therefore, by minimality, there exist $f_{1},\dots,f_{n}\in \cF^{[d]}_{*}$ such that $Z= \bigcup_{i=1}^{n}f_{i}^{-1}V$.

    Let $i\in \{1,\dots,n\}$. Since $f_{i}$ and $\sigma$ are uniformity continuous, it follows that there exists $\alpha_{i}\in \mathcal{U}_{X^{[d]}_{*}}$ such that $\sigma(f_{i}\bx_{1},f_{i}\bx_{2})<\varepsilon-c$ whenever $(\bx_{1},\bx_{2})\in \alpha_{i}$. By \cref{lemma: uniform_pseudometrics}, there exist $\rho_{i}\in \Sigma_{d}$ and $\delta_{i}>0$ such that $(\bx_{1},\bx_{2})\in \alpha_{i}$ whenever $\rho_{i}(\bx_{1},\bx_{2})<\delta_{i}$.

    Now, define $\rho = \max_{i=1,\dots,n}\rho_{i}$ and $\delta = \min_{i
    =1,\dots,n}\delta_{i}$. Let $z\in U(y,\rho,\delta)$, then there exists $f\in \cF^{[d]}_{*}$ such that $\rho(fy^{[d]}_{*},fz^{[d]}_{*})<\delta$. Therefore, we have $\sigma(f_{i}fy^{[d]}_{*},f_{i}fz^{[d]}_{*})<\varepsilon-c$ for all $i\in \{1,\dots,n\}$. Since $f(x^{[d]}_{*},y^{[d]}_{*})\in Z$, there exists $i\in \{1,\dots,n\}$ such that $f_{i}f(x^{[d]}_{*},y^{[d]}_{*}) \in V$. Thus, we obtain that\begin{align*}
        \sigma(f_{i}fx^{[d]}_{*},f_{i}fz^{[d]}_{*})\leq \sigma(f_{i}fx^{[d]}_{*},f_{i}fy^{[d]}_{*})+\sigma(f_{i}fy^{[d]}_{*},f_{i}fz^{[d]}_{*})<c+\varepsilon-c = \varepsilon
    \end{align*} 

    Hence, we get that $F_{\sigma,d}(x,z)<\varepsilon$, and we conclude that $U(y,\rho,\delta)\subseteq U(x,\sigma,\varepsilon)$.
\end{proof}

\begin{theorem}
    The family\begin{align*}
        \{U(x,\sigma,\varepsilon): x\in uX, \sigma\in \Sigma_{d},\varepsilon>0\}
    \end{align*}

    forms a basis for a topology on $uX$, which we call the $\mathscr{F}^{[d]}$-topology.
\end{theorem}

\begin{proof}
    Let $x_{1},x_{2}\in uX$, $\sigma_{1},\sigma_{2}\in \Sigma_{d}$ and $\varepsilon_{1},\varepsilon_{2}>0$ be such that $U(x_{1},\sigma_{1},\varepsilon_{1})\cap U(x_{2},\sigma_{2},\varepsilon_{2})\neq \emptyset$. Let $y\in U(x_{1},\sigma_{1},\varepsilon_{1})\cap U(x_{2},\sigma_{2},\varepsilon_{2})$, and define $\sigma = \max\{\sigma_{1},\sigma_{2}\}$ and $\varepsilon=\min\{\varepsilon_{1},\varepsilon_{2}\}$. Note that \begin{align*}
        U(y,\sigma,\varepsilon) \subseteq U(x_{1},\sigma_{1},\varepsilon_{1})\cap U(x_{2},\sigma_{2},\varepsilon_{2}),
    \end{align*}

    completing the proof. 
\end{proof}

Now, to prove that the $\mathscr{F}^{[d]}$-topology coincides with the $\uptau^{[d]}$-topology, we need the following lemma.

\begin{lemma}\label{lemma: convergence_in_furstenberg_d_topology}
    If a net $(\bx_{i})_{i\in I}$ in $u^{[d]}_{*}X^{[d]}_{*}$ and a point $\bx$ in $u^{[d]}_{*}X^{[d]}_{*}$ satisfy that $F^{[d]}_{\sigma}(\bx_{i},\bx)\to 0$ for every $\sigma\in\Sigma_{d}$ then there exist a subnet $(\bx_{j})_{j\in J}$ of $(\bx_{i})_{i\in I}$ and a net $(f_{j})_{j\in J}$ in $\cF^{[d]}_{*}$ such that $\sigma(f_{j}\bx_{j},f_{j}\bx)\to 0$ for every $\sigma\in \Sigma_{d}$.
\end{lemma}

\begin{proof}
    Let $J = \Sigma_{d}\times \R_{+}\times I$, note that $J$ is a directed set, where $\R_{+}$ is endowed with the inverse order (that $\epsilon_{1}\succ \epsilon_{2}$ if $\epsilon_{1}\leq \epsilon_{2}$). For each $j = (\sigma,\epsilon,i)\in J$, choose $\phi(j)\geq i$ in $I$ such that $F^{[d]}_{\sigma}(\bx_{\phi(j)},\bx)<\epsilon$. Then there exists $f_{j}\in \cF^{[d]}_{*}$ such that $\sigma(f_{j}\bx_{\phi(j)}, f_{j}\bx)<\epsilon$. Let $\sigma'\in \Sigma_{d}$ and $\epsilon'>0$, for all $i\in I$ we have\begin{align*}
        \sigma'(f_{j}\bx_{\phi(j)}, f_{j}\bx)\leq \sigma(f_{j}\bx_{\phi(j)}, f_{j}\bx)<\epsilon<\epsilon'
    \end{align*}

    for every $j=(\sigma,\epsilon,i)\succ (\sigma',\epsilon',i)$. This proves the desired result.
\end{proof}

\begin{remark}
    Note that, directly from the definition of the $\mathscr{F}^{[d]}$-topology, if $(x_{i})_{i}$ is a net in $uX$ and $(f_{i})_{i}$ is a net in $\cF^{[d]}_{*}$, and there exists $x\in uX$ such that $\sigma(f_{i}x_{i*}^{[d]},f_{i}x^{[d]}_{*})\to 0$ for every $\sigma\in \Sigma_{d}$ then the net $(x_{i})_{i}$ converges to $x$ with respect to the $\mathscr{F}^{[d]}$-topology.
\end{remark}

\begin{theorem}
    The $\uptau^{[d]}$-topology coincides with the $\mathscr{F}^{[d]}$-topology on $uX$.
\end{theorem}

\begin{proof}
    Let $A\subseteq uX$ and let $x\in \mathrm{cl}_{\mathscr{F}^{[d]}}(A)$, where $\mathrm{cl}_{\mathscr{F}^{[d]}}$ denotes the closure operator with respect to the $\mathscr{F}^{[d]}$-topology. Then, there exists a net $(x_{i})_{i\in I}$ in $A$ convergent to $x\in uX$ with respect to the $\mathscr{F}^{[d]}$-topology. By \cref{lemma: convergence_in_furstenberg_d_topology}, there exist a subnet $(x_{j})_{j\in J}$ of $(x_{i})_{i\in I}$ and a net $(f_{j})_{j\in J}$ in $\cF^{[d]}_{*}$ such that $\sigma(f_{j}x_{j*}^{[d]},f_{j}x^{[d]}_{*})\to 0$ for every $\sigma\in \Sigma_{d}$.

    Passing to a subnet if necessary, assume that $f_{j}x_{j*}^{[d]}\to \bx$ and $f_{j}\to f$ for some $\bx\in X^{[d]}_{*}$ and $f\in\beta\cF^{[d]}_{*}$. Therefore, we have $\sigma(\bx,fx^{[d]}_{*})=0$ for every $\sigma\in \Sigma_{d}$, and hence $\bx=fx^{[d]}_{*}$. In particular, we get that $x^{[d]}_{*} = (u^{[d]}_{*}fu^{[d]}_{*})^{-1}\bx$. Moreover, since $\bx\in fu^{[d]}_{*}\circ_{\cF^{[d]}_{*}} \Delta^{[d]}_{*}(A_{k})$ for every $k\in J$, where $A_{k} = \{x_{j}:j\geq k\}$, it follows that\begin{align*}
        x^{[d]}_{*} \in (u^{[d]}_{*}fu^{[d]}_{*})^{-1}(fu^{[d]}_{*}\circ_{\cF^{[d]}_{*}} \Delta^{[d]}_{*}(A_{k})) \subseteq (u^{[d]}_{*}fu^{[d]}_{*})^{-1}fu^{[d]}_{*}\circ_{\cF^{[d]}_{*}} \Delta^{[d]}_{*}(A_{k}) = u^{[d]}_{*} \circ_{\cF^{[d]}_{*}} \Delta^{[d]}_{*}(A_{k}).
    \end{align*}

    Hence, after passing to a subnet if necessary, we obtain $x_{j}\taulimF{d} x$, which implies that  $x\in \cltauF{d}{A}$. Therefore, we conclude that $\mathrm{cl}_{\mathscr{F}^{[d]}}(A) \subseteq \cltauF{d}{A}$.

    Conversely, let $x\in \cltauF{d}{A}$. Then there exist a net $(x_{i})_{i\in I}$ in $A$ and a net $(f_{i})_{i\in I}$ in $\cF^{[d]}_{*}$ such that $f_{i}\to u^{[d]}_{*}$ and $f_{i}x_{i*}^{[d]}\to x^{[d]}_{*}$. Therefore, we have $\sigma(f_{i}x_{i*}^{[d]},f_{i}x^{[d]}_{*})\to 0$ for every $\sigma \in \Sigma_{d}$, and thus $(x_{i})_{i\in I}$ converge to $x$ with respect to the $\mathscr{F}^{[d]}$-topology. Hence, we conclude that $\cltauF{d}{A}\subseteq \mathrm{cl}_{\mathscr{F}^{[d]}}(A)$.
\end{proof}

\subsection{An algebraic characterization of $\RP^{[d]}$}\label{subsec: algebraic_RPd}

In this subsection we show an algebraic characterization of $\RP^{[d]}$, for abelian actions, using the $\uptau^{[d]}$-topology. Throughout, we assume that in every system $(X,T)$, the acting groups $T$ is an abelian group.

\begin{theorem}\label{thm: algebraic_characterization_RPd}
    Let $(X,T)$ be a minimal system and let $d\geq 1$ be an integer. Then\begin{align*}
        \RP^{[d]}(X)= \{(x,vgx):x\in X,v\in J(M),g\in H_{\uptau^{[d]}}(G)\}.
    \end{align*}
\end{theorem}

\begin{proof}
    By \cref{prop: H_taud_NRPd} and \cref{thm: NRP_properties}, it is enough to prove that \begin{align*}
        \RP^{[d]}(X) \subseteq \{(x,vgx):x\in X,v\in J(M),g\in H_{\uptau^{[d]}}(G)\}.
    \end{align*}

    We first consider the case $X=M$. Let $(g,u)\in \RP^{[d]}(M)$ with $g\in G$. Then there exists $(f_{i})_{i\in I}$ a net in $\cF^{[d]}$ and $(g_{i})_{i\in I}$ a net in $M$ such that\begin{align*}
        (f_{i}g_{i}^{[d]},f_{i}u^{[d]})\to (g,u^{[d]}_{*},u,u^{[d]}_{*}).
    \end{align*}

    Since $G$ is dense in $M$, we may assume that $(g_{i})_{i\in I} \subseteq G$. By \cref{lemma: xi_to_x_xi_taulim_x}, we have $g_{i}\taulimF{d} g$. Let $\sigma$ be a pseudometric on $M^{[d]}_{*}$. Since $F_{\sigma}^{[d]}$ is upper semicontinuous and $\cF^{[d]}_{*}$-invariant, it follows that \begin{align*}
       \limsup F_{\sigma,d}(g_{i},u) =\limsup F_{\sigma}^{[d]}(f_{i*}g_{i*}^{[d]},f_{i*}u^{[d]}_{*}) \leq F^{[d]}_{\sigma}(u^{[d]}_{*},u^{[d]}_{*})=0.
    \end{align*}

    Thus, we obtain that $g_{i}\taulimF{d} u$, and hence $g\in H_{\uptau^{[d]}}(G)$.

    Now, let $(p,q)\in \RP^{[d]}(M)$, note that $(uq^{-1}p,u)\in \RP^{[d]}(M)$ since $\RP^{[d]}(M)$ is invariant. Therefore, we have $up\in uqH_{\uptau^{[d]}}(G)$, and we conclude that $p=vqg$ for some $v\in J(M)$ and $g\in H_{\uptau^{[d]}}(G)$.

    For the general case, it follows from the fact that $\pi\times\pi(\RP^{[d]}(M))=\RP^{[d]}(X)$, where $\pi:M\to X$ is a factor map.
\end{proof}

A direct consequence is the following algebraic characterization of systems of order $d$.

\begin{corollary}\label{cor: Ellis_group_factor_order_d}
    Let $(X,T)$ be a minimal system and let $d\geq 1$ be an integer. Then, $(X,T)$ is a system of order $d$ if and only if $(X,T)$ is a distal system and $H_{\uptau^{[d]}}(G)\subseteq \GG(X)$. Moreover, $\GG(X/\RP^{[d]}(X))=H_{\uptau^{[d]}}(G)\GG(X)$.
\end{corollary}

Using the algebraic characterization, we give an alternative proof of the following theorem without using the gluing property of the dynamical cube in the distal case.

\begin{theorem}
    Let $d\geq 2$ be an integer and let $(X,T)$ be a system of order $d$. Then the following is a sequence of group factors:
    \[\begin{tikzcd}[cramped,column sep=tiny]
    	{(X,T)} && {(X/\RP^{[d-1]}(X),T)} && \cdots && {(X/\RP^{[1]}(X),T)} && {\{\cdot\}}
    	\arrow[from=1-1, to=1-3]
    	\arrow[from=1-3, to=1-5]
    	\arrow[from=1-5, to=1-7]
    	\arrow[from=1-7, to=1-9]
    \end{tikzcd}\]
\end{theorem}

\begin{proof}
    We first show that $\pi_{d-1}: X\to X/\RP^{[d-1]}(X)$ is an equicontinuous factor. Let $(x,y)\in Q(\RP^{[d-1]}(X))$, and let $\alpha \in \mathcal{U}_{X}$. Consider $\beta \in \mathcal{U}_{X}$ such that $\beta \subseteq \alpha$ and $\beta^{2}\subseteq \alpha$. Since $(x,y) \in Q(\RP^{[d-1]}(X))$, there exist $(x_{1},y_{1})\in \RP^{[d-1]}(X)$ and $t\in T$ such that $(x,x_{1}),(y,y_{1}),(tx_{1},ty_{1})\in \beta$.

    Let $\gamma\in \mathcal{U}_{X}$, with $\gamma\subseteq\beta$, be such that $(a,b)\in \gamma$ implies $(ta,tb)\in \alpha$. Since $(x_{1},y_{1})\in \RP^{[d-1]}(X)$, there exist $x_{2},y_{2}\in X$ and $f\in \cF^{[d-1]}(X)$ such that $(x_{1},x_{2}),(y_{1},y_{2}),(f_{\epsilon}x_{2},f_{\epsilon}y_{2}) \in \gamma$ for each $\epsilon \neq \emptyset$. Therefore, we have $(x,x_{2}), (y,y_{2})\in \beta^{2}$ and $(tf_{\epsilon}x_{2},tf_{\epsilon}y_{2})\in \alpha$ for each $\epsilon \neq \emptyset$. Define $f' = (f,t^{[d-1]}f)\in \cF^{[d]}$. Then $(x,y_{2}),(y,y_{2}), (f'_{\epsilon}x_{2},f'_{\epsilon}y_{2})\in \alpha$ for each $\varepsilon\neq \emptyset$, and hence $(x,y)\in \RP^{[d]}(X)=\Delta$. Thus, we conclude that $\pi_{d-1}$ is equicontinuous.

    Note that, by \cref{thm: algebraic_characterization_RPd}, we have \begin{align*}
        H_{\uptau^{[d-1]}}(G) = \{g\in G: (u,g)\in \RP^{[d-1]}(M)\}.
    \end{align*}

    In particular, we deduce that $H_{\uptau^{[d-1]}}(G)$ is a normal subgroup. Furthermore, since $H_{\uptau^{[d-1]}}(G)$ is a normal subgroup and $X$ is a distal system, by \cref{thm: extensions_ellis_group} and \cref{cor: Ellis_group_factor_order_d}, it follows that $\pi_{d-1}$ is a group factor.
\end{proof}

We expect that \cref{thm: algebraic_characterization_RPd} also holds for systems satisfying the Bronstein condition. This would imply that $\RP^{[d]}$ is an equivalence relation for such systems, thereby providing a positive answer to a question of Gutman, Glasner, and Ye {\cite[Question 10.5]{Glasner_Gutman_Ye_higher_regionallyproximal_general_groups:2018}}. The main difficulties in proving this using the $\uptau^{[d]}$-topology, rather than the $\uptau$-topology, are the following. First, right multiplication and inversion are not homeomorphisms with respect to the $\uptau^{[d]}$-topology. Second, when $d>2$, one has $(\pi^{-1}(x_{0}))^{[d]}_{*} = \GG(X)^{[d]}_{*}$ instead of $\Delta^{[d]}_{*}(\GG(X))$, where $\pi: M\to X$ is the natural factor map, this creates a difficulty, since the definition of the $\uptau^{[d]}$-topology requires elements of the form $g^{[d]}_{*}$.

\subsection{A conjecture on $\mathcal{A}(\RP^{[d]})$ and its consequences}

In this subsection, we state a conjecture that gives an algebraic description of the smallest closed invariant equivalence relation containing $\RP^{[d]}$ for any group action, not necessarily abelian. We also show some consequences, assuming that the conjecture is true.

Let $R$ be a relation on a system $(X,T)$, and $\mathcal{A}(R)$ denote the smallest closed, invariant equivalence relation containing $R$. 

Every system $(X,T)$ possesses a maximal distal factor. Moreover, there exists a closed, invariant equivalence relation $S_{dis}(X)$ such that the quotient $X/S_{dis}(X)$ is the maximal distal factor of $X$, $S_{dis}(X)$ is called the {\em distal structure relation} of $X$. Similarly, for every system $(X,T)$ and $d\geq 1$, there exists a closed, invariant equivalence relation $S_{d}(X)$ such that the quotient $X/S_{d}(X)$ is the maximal factor of order $d$ of $X$. Moreover, it is straightforward to see that $S_{d}(X)=\mathcal{A}(\RP^{[d]}(X))$, and consequently $M/\mathcal{A}(\RP^{[d]}(M))$ is the maximal minimal system of order $d$. 

Let $M_{dis}$ be the maximal distal factor of $M$ and let $D=\GG(M_{dis})$. The group $D$ is a $\uptau$-closed normal subgroup of $G$ and is generated by $\{g\in G: (p,gp)\in \overline{P(M)} \text{ for each }p\in M\}$. Moreover, $D=\{g\in G: (p,gp)\in S_{dis}(M)\}$. For more details on this group, see \cite{Auslander_Glasner_distal_order:2002}.

\begin{conjecture}\label{conj: characterization_ARPd}
    Let $(X,T)$ be a minimal system and let $d\geq 1$ be an integer. Then,\begin{align*}
        \mathcal{A}(\RP^{[d]}(X)) = \{(x,vhgx):x\in X, h\in H_{\uptau^{[d]}}(G), g\in D,v\in J(M)\}.
    \end{align*}
\end{conjecture}

The conjecture is known to be true in the case $d=1$. This follows from the results in \cite{Auslander_Glasner_distal_order:2002,Veech_topological_dynamics:1977}.

To prove \cref{conj: characterization_ARPd}, it is enough to prove it in the distal case. More precisely, if $(X,T)$ is a distal minimal system, then it is enough to show that\begin{align*}
        \mathcal{A}(\RP^{[d]}(X)) = \{(x,vhgx):x\in X, h\in H_{\uptau^{[d]}}(G),v\in J(M)\}.
\end{align*}

\begin{theorem}\label{thm: algebraic_A_RPd}
    Let $(X,T)$ be a minimal system and let $d\geq 1$ be an integer. Assume that $\cref{conj: characterization_ARPd}$ holds for distal systems. Then,\begin{align*}
        \mathcal{A}(\RP^{[d]}(X)) = \{(x,vhgx):x\in X, h\in H_{\uptau^{[d]}}(G), g\in D,v\in J(M)\}.
    \end{align*}
\end{theorem}

\begin{proof}
    Let $(x,y)\in \mathcal{A}(\RP^{[d]}(X))$. Since $X_{dis}/\mathcal{A}(\RP^{[d]}(X_{dis}))$ and $X/\mathcal{A}(\RP^{[d]}(X))$ are isomorphic, it follows that $(u\pi(x),u\pi(y))\in \mathcal{A}(\RP^{[d]}(X_{dis}))$, where $\pi:X\to X_{dis}$ is the quotient map. Since \cref{conj: characterization_ARPd} is true for $X_{dis}$, there exists $h\in H_{\uptau^{[d]}}(G)$ such that $u\pi(y) = uh\pi(x)$. Since $(uy,uhx)\in S_{dis}(X)$, there exists $r\in D$ such that $uy = uhrx$. Now, let $v\in J(M)$ be such that $vy=y$, so we conclude that \begin{align*}
        y=v u y = vuhrx = vhrx.
    \end{align*}

    Now, let $x\in X$, $h\in H_{\uptau^{[d]}}(G)$, $g\in D$ and $v\in J(M)$. By \cref{prop: H_taud_NRPd}, it follows that $(ux,hx)\in \mathcal{A}(\RP^{[d]}(X))$. Furthermore, we have $(x,ux),(hx,ghx),(ghx,vghx)\in S_{dis}(X)\subseteq \mathcal{A}(\RP^{[d]}(X))$. Hence, since $\mathcal{A}(\RP^{[d]}(X))$ is an equivalence relation, we conclude that $(x,vdhx)\in \mathcal{A}(\RP^{[d]}(X))$.
\end{proof}

We now show some consequences, assuming that \cref{conj: characterization_ARPd} is true. Throughout, we assume that \cref{conj: characterization_ARPd} holds.

In this case, \cref{thm: algebraic_A_RPd} gives a proof of the lifting property of $\mathcal{A}(\RP^{[d]})$. This is interesting because, for general group actions, it is still not known whether $\RP^{[d]}$ itself has the lifting property, even assuming \cref{conj: characterization_ARPd}.

\begin{theorem}
    Let $\pi:(X,T)\to (Y,T)$ be a factor map between minimal systems and let $d\geq 1$ be an integer. Then $\pi\times\pi(\mathcal{A}(\RP^{[d]}(X)))=\mathcal{A}(\RP^{[d]}(Y))$.
\end{theorem}

\begin{proof}
    It is clear that $\pi\times\pi(\mathcal{A}(\RP^{[d]}(X))) \subseteq \mathcal{A}(\RP^{[d]}(Y))$, so we only need to prove the reverse inclusion. Let $(y,y')\in \mathcal{A}(\RP^{[d]}(Y))$. By \cref{thm: algebraic_A_RPd}, there exists $v\in J(M)$, $g\in D$ and $h\in H_{\uptau^{[d]}}(G)$ such that $y'=vhgy$. Consider $x\in X$ such that $y=\pi(x)$. Then, by \cref{thm: algebraic_A_RPd} we have $(x,vhgx)\in \mathcal{A}(\RP^{[d]}(X))$, and $\pi(vhgx)=y'$. This finishes the proof.
\end{proof}

Another consequence is the following characterization of $\mathcal{A}(\RP^{[d]})$.

\begin{corollary}
    Let $(X,T)$ be a minimal system and let $d\geq 1$ be an integer. Then $\mathcal{A}(\RP^{[d]}(X)) = \RP^{[d]}(X)S_{dis}(X)$.
\end{corollary}

\begin{proof}
    Clearly, $\RP^{[d]}(X)S_{dis}(X) \subseteq \mathcal{A}(\RP^{[d]}(X))$. Let $(x,y)\in \mathcal{A}(\RP^{[d]}(X))$. By \cref{thm: algebraic_A_RPd}, there exist $h\in H_{\uptau^{[d]}}(G)$, $g\in D$ and $v\in J(M)$ such that $y=vhgx$. It follows from \cref{prop: H_taud_NRPd} that $(x,whx)\in \RP^{[d]}(X)$, where $w\in J(M)$ satisfies $wx=x$. Since $(whx,vhgx)\in S_{dis}(X)$, we conclude that $(x,y)\in \RP^{[d]}(X)S_{dis}(X)$.
\end{proof}

Furthermore, \cref{thm: algebraic_A_RPd} gives explicitly the Ellis group of the maximal factor of order $d$ for any minimal system.

\begin{corollary}
    Let $(X,T)$ be a minimal system and let $d\geq 1$. Then $\GG(X_{d})=H_{\uptau^{[d]}}(G)D\GG(X)$, where $X_{d}$ is the maximal factor of order $d$.
\end{corollary}

A key tool in the following applications is the capturing operation, a kind of reverse orbit closure, which was introduced in \cite{Auslander_Glasner_distal_order:2002} to characterize the distal and equicontinuous structure relations in minimal systems.

Let $(X, T)$ be a minimal topological dynamical system and $K\subseteq X$. The capturing set of $K$ is $C(K)=\{x\in X: \overline{Tx}\cap K \neq \emptyset\}$. We say that $K$ is a capturing set if $C(K)=K$. Let $R$ be a symmetric and reflexive relation, and let $\mathcal{E}(R)$ be the equivalence relation generated by $R$. Thus, $\mathcal{E}(R)=\bigcup\{R^{n}:n=1,2,\dots\}$, where\begin{align*}
    R^{n}=\{(x,z):\exists y_{1},\dots,y_{n-1} \text{ such that } (x,y_{1}),(y_{1},y_{2}),\dots,(y_{n-1},z)\in R \}.
\end{align*}

If $R$ and $S$ are relations, we also define\begin{align*}
    RS = \{(x,z):(x,y)\in R, (y,z)\in S\text{ for some }y\}
\end{align*}

We will need the following property of the distal structure relation.

\begin{theorem}[{\cite[Theorem 3.3]{Auslander_Glasner_distal_order:2002}}]\label{thm: S_dis_capturing_smallest_diagonal}
    Let $(X,T)$ be a minimal system. Then $S_{dis}(X)$ is the smallest closed capturing set containing $\Delta$.
\end{theorem}

Now, we show that $\mathcal{A}(\RP^{[d]}(X))$ is a capturing set for any minimal system $(X,T)$. Before we need the following lemma.

\begin{lemma}[{\cite[Lemma 1.2]{Auslander_Glasner_distal_order:2002}}]\label{lemma: capturing_R_PR}
    Let $(X,T)$ be a system and $R$ be a closed invariant subset of $X\times X$. Then $\P(X)\, R\subseteq C(R)\subseteq \P(X)\,R\,\P(X)$. If $(X,T)$ is minimal, then $C(R)=\P(X)R$
\end{lemma}

\begin{proposition}
    Let $(X,T)$ be a minimal system and let $d\geq 1$ be an integer. Then $\mathcal{A}(\RP^{[d]}(X))$ is a capturing set.
\end{proposition}

\begin{proof}
    By \cref{lemma: capturing_R_PR} and $\P(X)\subseteq \RP^{[d]}(X)$, it follows that \begin{align*}
        \mathcal{A}(\RP^{[d]}(X)) = \P(X)\mathcal{A}(\RP^{[d]}(X)) \subseteq C(\mathcal{A}(\RP^{[d]}(X))) \subseteq \P(X)\mathcal{A}(\RP^{[d]}(X))\P(X) = \mathcal{A}(\RP^{[d]}(X)).
    \end{align*}
\end{proof}

Let $(X,T)$ be a minimal system and let $d\geq 1$ be an integer. We denote by $\RP^{[d]}_{\#}(X)$ the smallest closed capturing set containing $\RP^{[d]}(X)$. We define\begin{align*}
    K_{d}(X)= \{(p,q)\in M\times M: \pi\times\pi(p,qh)\in \RP^{[d]}_{\#}(X) \text{ for each }h\in H_{\uptau^{[d]}}(G)\},
\end{align*} 

where $\pi:M\to X$ is a factor map. We will now show that $\RP^{[d]}_{\#}(X) = \mathcal{A}(\RP^{[d]}(X))$. Before we need the following lemmas.

\begin{lemma}\label{lemma: K_d_contain_S_dis}
    Let $(X,T)$ be a minimal system and let $d\geq 1$ be an integer. Then $K_{d}(X)$ is a closed capturing set containing $S_{dis}(M)$.
\end{lemma}

\begin{proof}
    Since $\RP^{[d]}_{\#}(X)$ is a closed capturing set, we have that $K_{d}(X)$ is a closed capturing set. By \cref{thm: S_dis_capturing_smallest_diagonal}, to show that $S_{dis}(M)\subseteq K_{d}(X)$ it is enough to prove that $\Delta \subseteq K_{d}(X)$. Let $p\in M$ and $h\in H_{\uptau^{[d]}}(G)$. It follows from \cref{prop: H_taud_NRPd} that $(px_{0},phx_{0})\in \RP^{[d]}(X)$, and hence $(p,p)\in K_{d}(X)$.
\end{proof}

\begin{lemma}[{\cite[Proposition 1.1]{Auslander_Glasner_distal_order:2002}}]\label{lemma: x_capturing_iff_prox_dom_in_k}
    Let $(X,T)$ be a system and let $K$ be a closed invariant set. Then $x\in C(K)$ if and only if $x$ is proximal to some point of $K$.
\end{lemma}

\begin{theorem}\label{thm: RPd_capturing}
    Let $(X,T)$ be a minimal system and let $d\geq 1$ be an integer. Then $\mathcal{A}( \RP^{[d]}(X))=\RP^{[d]}_{\#}(X)$.
\end{theorem}

\begin{proof}
    Since $\mathcal{A}( \RP^{[d]}(X))$ is a closed capturing set containing $\RP^{[d]}(X)$, we have $\RP^{[d]}_{\#}(X)\subseteq \mathcal{A}( \RP^{[d]}(X))$. Now, let $(x,y)\in \mathcal{A}( \RP^{[d]}(X))$ be a minimal point, and let $p\in M$ such that $x=px_{0}$. It follows from \cref{thm: algebraic_A_RPd} that there exists $h\in H_{\uptau^{[d]}}(G)$ and $g\in D$ such that $y=ghx$. By \cref{lemma: K_d_contain_S_dis}, we have $(p,pg)\in S_{dis}(M)\subseteq K_{d}(X)$. Therefore, we have that\begin{align*}
        (x,y) = (px_{0},pghx_{0})\in \RP^{[d]}_{\#}(X).
    \end{align*}

    Thus, we obtain that all minimal points of $\mathcal{A}(\RP^{[d]}(X))$ are in $\RP^{[d]}_{\#}(X)$. Since $\RP^{[d]}_{\#}(X)$ is a capturing set, by \cref{lemma: x_capturing_iff_prox_dom_in_k}, we conclude that $\mathcal{A}(\RP^{[d]}(X))\subseteq \RP^{[d]}_{\#}(X)$.
\end{proof}

A direct consequence is the following result, which gives a necessary and sufficient condition for $\RP^{[d]}(X)$ to be an equivalence relation.

\begin{corollary}
    Let $(X,T)$ be a minimal system and let $d\geq 1$ be an integer. Then $\RP^{[d]}(X)$ is an equivalence relation if and only if $\RP^{[d]}(X) = \P(X)\RP^{[d]}(X)$.
\end{corollary}

\section{Recurrence sets and $\NRP^{[d]}$}

\subsection{Correlation for finitely generated abelian actions}\label{subsec: correlation}

In this subsection, we assume that in every system $(X,T)$, the acting group $T$ is a finitely generated abelian group.

Let $H$ be a countable topological group, written with multiplicative notation. A {\em F{\o}lner sequence} in $H$ is a sequence $\Phi = (\Phi_{N})_{N\in \N}$ of nonempty finite subsets of $H$  such that for every $h\in H$,\begin{align*}
    \lim_{N\to \infty}\dfrac{|h\Phi_{N}\Delta \Phi_{N}|}{|\Phi_{N}|} =0.
\end{align*}

\begin{definition}
    Let $d\geq 1$ be an integer and let $(X=L/\Gamma,T)$ be a $d$-step nilsystem. Let $\phi$ be a continuous function on $X$ and $b\in X$. The sequence $\{\phi(tb)\}_{t\in T}$ is called a {\em basic $d$-step nilsequence}. A {\em $d$-step nilsequence} is a uniform limit of basic $d$-step nilsequences.
\end{definition}

\begin{definition}
    Let $\{a_{g}\}_{g\in H}$ be a bounded sequence indexed by elements of a countable amenable group $H$. We say that $(a_{g})_{g\in H}$ is a {\em null-sequence} if it tends to zero in uniform density, that is, \begin{align*}
        \lim_{N\to\infty}\dfrac{1}{|\Phi_{N}|}\sum_{g\in H}|a_{g}| =0
    \end{align*}

    for any F{\o}lner $\Phi$ of $H$.
\end{definition}

In this subsection, we show that for a finitely generated abelian group $T$, the sequence $(c(t))_{t\in T}$ is the sum of a null-sequence and a $d$-step nilsequence, where \begin{align*}
    c(t) = \int f(x) \cdot f(tx)\cdot \ldots \cdot f(t^{d}x) \, d\mu(x).
\end{align*}

We follow the approach of \cite{Bergelson_Host_Kra05,Leibman15}. 

In \cite{Leibman15}, Leibman studied multiparameter polynomial multiple correlations for $\mathbb{Z}^{d}$-actions. The results of that paper also hold for actions of finitely generated abelian groups. In particular, we will use the following decomposition result for nilsystems.

 \begin{theorem}[{\cite[Theorem 4.3]{Leibman15}}]\label{thm: nil_connnected_descomp_nil_null}
     Let $(X,T)$ be a $d$-step nilsystem, let $Y$ be a subnilmanifold of $X$ and let $f\in C(X)$. Then the sequence $c(t) = \int_{tY} f d\mu_{tY}$, $t\in T$, can be decomposed into a sum of a null-sequence and a $d$-step basic nilsequence.
 \end{theorem}

We will also use the following result.

\begin{lemma}[{\cite[Lemma 4.2]{Bergelson_McCutcheon_Zhang_Roth_thm_amenable:1997}}]\label{lemma: van_der_corput_nilpotent}
    Let $(a_{g})_{g\in H}$ be a bounded family of elements of a Hilbert space indexed by elements of a countable amenable group $H$ and let $\Phi$ be a F\o lner sequence in $H$. If 
    \begin{align*}
        \limsup_{N\to \infty} \dfrac{1}{|\Phi_{N}|^{2}}\left( \limsup_{M\to \infty} \dfrac{1}{|\Phi_{M}|}\sum_{g\in \Phi_{M}}\sum_{h,k\in \Phi_{N}} \langle a_{hg},a_{kg} \rangle \right) = 0
    \end{align*}
    Then\begin{align*}
        \limsup_{N\to \infty} \left\| \dfrac{1}{|\Phi_{N}|}\sum_{g\in \Phi_{N}}a_{g}\right\|^{2}=0
    \end{align*}
\end{lemma}

Using this version of the van der Corput lemma, we obtain the following result, which is well known for $\Z$-actions  ({\cite[Theorem 12.1]{Host_Kra_nonconventional_averages_nilmanifolds:2005}}).

\begin{theorem}\label{thm: nil_characteristic}
    Let $(X,\cX,\mu,T)$ be an ergodic measure preserving system, let $d\geq 1$ be an integer and let $f_{0},f_{1},\dots,f_{d}$ be bounded functions on $X$. If at least one of these functions has zero conditional expectation on $\cZ_{d}$ then\begin{align*}
        \lim_{N\to \infty} \left\| \dfrac{1}{|\Phi_{N}|}\sum_{t\in \Phi_{N}} f_{0}\cdot f_{1}\circ t\cdot f_{2}\circ t^{2} \cdots f_{d}\circ t^{d} \right\|_{L^{2}(\mu)}=0
    \end{align*}

    for any F\o lner $\Phi$ of $T$.
\end{theorem}

\begin{proof}
    We proceed by induction. For $d=1$, assume, without loss of generality, that $\nnorm{f_{1}}_{1}=0$, by the ergodic theorem, we have\begin{align*}
        \left\| \dfrac{1}{|\Phi_{N}|}\sum_{t\in \Phi_{N}}  f_{0} \cdot f_{1}\circ t \right\|_{L^{2}(\mu)} \leq \|f_{0}\|_{L^{\infty}(\mu)} \left\| \dfrac{1}{|\Phi_{N}|}\sum_{t\in \Phi_{N}} f_{1}\circ t \right\|_{L^{2}(\mu)} \to  \|f_{0}\|_{L^{\infty}(\mu)}\nnorm{f_{1}}_{1}=0
    \end{align*}

    Now, let $d\geq 1$ and assume that the statement is true for $d$. Let $f_{0},f_{1},\dots,f_{d+1}\in L^{\infty}(\mu)$. Set \begin{align*}
        u_{t} =  \prod_{j=0}^{d+1} f_{j}\circ t^{j}.
    \end{align*}

    By \cref{lemma: van_der_corput_nilpotent}, it suffices to prove that\begin{align*}
        \limsup_{M\to \infty} \dfrac{1}{|\Phi_{M}|}\sum_{t\in \Phi_{M}} \langle u_{ht},u_{kt} \rangle = 0
    \end{align*}

    for every $k,h\in T$.

    Let $k,h\in T$,\begin{align*}
        \left| \dfrac{1}{|\Phi_{M}|}\sum_{t\in \Phi_{M}} \langle u_{ht},u_{kt} \rangle \right| &= \left| \dfrac{1}{|\Phi_{M}|}\sum_{t\in \Phi_{M}} \int \prod_{j=0}^{d+1} f_{j}(h^{j}t^{j}x) \cdot f_{j}(k^{j}t^{j}x)\, d\mu(x)  \right|\\
        &= \left| \int \left(f_{0}(hx)\cdot f_{0}(kx)\right) \cdot \dfrac{1}{|\Phi_{M}|}\sum_{t\in \Phi_{M}} \left( \prod_{j=1}^{d+1}( f_{j} \circ h^{j} \cdot f_{j}\circ k^{j})(t^{j}x)\right)\,d\mu(x) \right|\\
        &\leq \|f_{0} \circ h \cdot f_{0} \circ k \|_{L^{2}(\mu)} \cdot\left\| \dfrac{1}{|\Phi_{M}|}\sum_{t\in \Phi_{M}} \left( \prod_{j=1}^{d+1} (f_{j} \circ h^{j} \cdot f_{j}\circ k^{j}) \circ t^{j}\right)\right\|_{L^{2}(\mu)}\\
        &= \|f_{0} \circ h \cdot f_{0} \circ k \|_{L^{2}(\mu)} \cdot\left\| \dfrac{1}{|\Phi_{M}|}\sum_{t\in \Phi_{M}} \left( \prod_{j=1}^{d+1} (f_{j} \circ h^{j} \cdot f_{j}\circ k^{j}) \circ t^{j-1}\right)\right\|_{L^{2}(\mu)}
    \end{align*}

    and by the inductive assumption, \begin{align*}
        \left\| \dfrac{1}{|\Phi_{M}|}\sum_{t\in \Phi_{M}} \left( \prod_{j=1}^{d+1} (f_{j} \circ h^{j} \cdot f_{j}\circ k^{j}) \circ t^{j-1}\right)\right\|_{L^{2}(\mu)} \to 0,
    \end{align*}

    completing the proof.
\end{proof}

From this theorem we obtain the following result. 

\begin{corollary}\label{cor: correlation_zero_unif_hostkra_zero}
    Let $(X,\mu,T)$ be an ergodic measure preserving system, let $d\geq 1$ be an integer and let $f_{0},f_{1},\dots,f_{d}$ be bounded functions on $X$. If at least one of these functions has zero conditional expectation on $\cZ_{d-1}$ then\begin{align*}
        \int f_{0}(x)f_{1}(tx)f_{2}(t^{2}x)\cdots f_{d}(t^{d}x)\, d\mu(x)
    \end{align*} 

    converges to zero in uniform density.
\end{corollary}

\begin{proof}
    Let $\Phi$ be a F\o lner sequence of $T$ and let $\mu\times \mu = \int_{Z}\mu_{s} \, dm(s)$ be the ergodic decomposition $\mu\times \mu$ under $T$, where $Z$ is the Kronecker factor and $m$ is the Haar measure. Note that for $m$-almost every $s\in Z$, one of the functions $f_{0}\otimes f_{0},f_{1}\otimes f_{1},\dots, f_{d}\otimes f_{d}$ has zero conditional expectation on $\cZ_{d}(X\times X,\mu_{s},T)$, and therefore, by \cref{thm: nil_characteristic}, we have\begin{align*}
        \lim_{N\to \infty}\dfrac{1}{|\Phi_{N}|} \sum_{t\in \Phi_{N}} \int f_{0}(x)f_{0}(x')f_{1}(tx)f_{1}(tx')\dots f_{d}(t^{d}x)f_{d}(t^{d}x')\,d\mu_{s}(x,x') = 0.
    \end{align*}

    Integrating with respect to $s$, we obtain \begin{align*}
        \lim_{N\to \infty}\dfrac{1}{|\Phi_{N}|} \sum_{t\in \Phi_{N}}\left( \int f_{0}(x)f_{1}(tx)\dots f_{d}(t^{d}x)\,d\mu(x)\right)^{2} = 0.
    \end{align*}
\end{proof}

Finally, we can prove the main result of this subsection.

\begin{theorem}\label{thm: correlation_nilpotent_actions}
    Let $(X,\cX,\mu,T)$ be an ergodic measure preserving system, let $f\in L^{\infty}(\mu)$ and let $d\geq 1$ be an integer. The sequence $(c(t))_{t\in T}$ is the sum of a null-sequence and a $d$-step nilsequence, where

    \begin{align*}
    c_{f}(t) = \int f(x) \cdot f(tx)\cdot \ldots \cdot f(t^{d}x) \, d\mu(x).
    \end{align*}
\end{theorem}

\begin{proof}
    We follow the steps of the proof of {\cite[Theorem 1.9]{Bergelson_Host_Kra05}}.
    
    Let $\widetilde{f} = \E(f|\cZ_{d})$. By \cref{cor: correlation_zero_unif_hostkra_zero}, the sequence $(c_{f}(t)-c_{\widetilde{f}}(t))_{t\in T}$ converges to zero in uniform density. Therefore, it suffices to prove the theorem for the function $\widetilde{f}$, so we can assume that $f$ is measurable with respect to $Z_{d}(X)$. By \cref{thm: str_host_kra_factors}, $Z_{d}(X)$ is the inverse limit of a sequence of ergodic $d$-step nilsystems. Thus, there exists a factor $X'$ of $Z_{d}(X)$ such that \begin{align*}
        \| f- \E(f|\cX')\|_{L^{1}(\mu)} \leq 1/((d+1)r),
    \end{align*}

    where $r\geq 1$ is an integer. Consequently, we have $|c_{f}(t)-c_{f'}(t)|\leq 1/r$, where $t\in T$ and $f'=\E(f\mid \cX')$. Applying \cref{thm: nil_connnected_descomp_nil_null} to the nilmanifold $X'^{d+1}$, the diagonal subnilmanifold $Y = \{(x,\dots,x):x\in X'\}$ and the function $F(x_{0},x_{1},\dots,x_{d}) =f'(x_{0})\cdot f'(x_{1})\cdot f'(x_{2})\cdots f'(x_{d})$, we have that $c_{f'}$ can be decomposed as a sum of a $d$-step nilsequence and a null sequence. Therefore, \begin{align*}
        c_{f}(t) = a_{r}(t)+b_{r}(t)+c_{r}(t)
    \end{align*}

    where $|a_{r}|\leq 1/r$, $b_{r}$ is a null sequence and $c_{r}$ is a $d$-step basic nilsequence. From {\cite[Lemma 1.11]{Bergelson_Host_Kra05}}, it follows that $(c_{r})_{r}$ converges uniformly to some sequence $(c(t))_{t\in T}$, and we conclude the desired result.
\end{proof}

\subsection{$\NRP^{[d]}$ and recurrence sets}\label{subsec: NRPd_recurrence}

In this subsection, following the approach of \cite{Huang_Shao_Ye_nilbohr_automorphy:2016}, we study $\NRP^{[d]}$ via recurrence sets.

We first give the definitions of the recurrence sets that we will use.

\begin{definition}
    Let $d\geq 1$ be an integer and $T$ be a topological group.\begin{enumerate}[label=(\arabic*)]
        \item We say that $S\subseteq T$ is a set of {\em $d$-recurrence} if for every measure preserving system $(X,\cX,\mu,T)$ and for every $A\in \cX$ with $\mu(A)>0$, there exists $t\in S$ such that $\mu(A\cap tA\cap \dots \cap t^{d}A)>0$.
        \item We say that $S\subseteq T$ is a set of {\em $d$-topological recurrence} if for every minimal system $(X,T)$ and for every nonempty open subset $U$ of $X$, there exists $t\in S$ such that $U\cap tU\cap \dots\cap t^{d}U \neq \emptyset$.
        \item We say that $S\subseteq T$ is a Nil$_{d}$ Bohr$_{0}$-set, if there are a $d$-step nilsystem $(X,T)$, $x\in X$ and a neighborhood of $x$ such that $N(x,U)\subseteq S$, where $N(x,U)=\{t\in T: tx\in U\}$.
    \end{enumerate}
\end{definition}

Let $\sF_{\operatorname{Poi}_{d}}$ (resp. $\sF_{\operatorname{Bir}_{d}}$, $\sF_{d}$) be the family consisting of all sets of $d$-recurrence (resp. sets of $d$-topological recurrence, the sets of Nil$_{d}$ Bohr$_{0}$-set).

\begin{theorem}\label{thm: characterization_NRP_Nil_bohr_sets}
    Let $(X,T)$ be a minimal system, with $T$ being a finitely generated group, and let $d\geq 1$ be an integer. Then $(x,y)\in \NRP^{[d]}(X)$ if and only if $N(x,U)\in \sF_{d}^{*}$ for each neighborhood $U$ of $y$, that is that $N(x,U)\cap A\neq\emptyset$ for each Nil$_{d}$ Bohr$_{0}$-set $A$.
\end{theorem}

\begin{proof}
    First suppose that $N(x,U)\in \sF_{d}^{*}$ for each neighborhood $U$ of $y$. Let $(X_{d},T)$ be the maximal factor of order $d$ of $X$ and let $\pi:X\to X_{d}$ be the factor map. It follows from \cref{thm: structure_thm_NRP} that $X_{d}$ is isomorphic to an inverse limit of nilsystems of order $d$. Thus, since $N(x,U)\in \sF_{d}^{*}$, it follows that $N(x,U)\cap N(\pi(x),V)\neq \emptyset$ for any neighborhood $V$ of $\pi(x)$. Therefore, there exists a net $(t_{i})_{i\in I}$ in $T$ such that $t_{i}(x,\pi(x))\to (y,\pi(x))$. Thus, we have\begin{align*}
        \pi(y)=\pi(\lim_{i} t_{i}x) = \lim_{i} t_{i}\pi(x)=\pi(x),
    \end{align*}

    and hence $(x,y)\in \NRP^{[d]}(X)$.

    Now, suppose that $(x,y)\in \NRP^{[d]}(X)$ and $U$ is a neighborhood of $y$. Let $(Z,T)$ be a $d$-step nilsystem, let $\widetilde{z}\in Z$ and let $V$ be a neighborhood of $\widetilde{z}$. Define\begin{align*}
        W=\prod_{z\in Z} Z.
    \end{align*}

    Note that $(W,T)$ is a distal system since $(Z,T)$ is distal. Choose $\omega^{*}\in W$ such that $\omega^{*}(z)=z$ for all $z\in Z$ and set $Z_{\infty} = \overline{T\omega^{*}}$. Observe that $(Z_{\infty},T)$ is a minimal distal system. In particular, for any $\omega\in Z_{\infty}$, there exists $g\in G$ such that $\omega(z)=g\omega^{*}(z)=gz$, and consequently $\omega(g^{-1}\widetilde{z})=\widetilde{z}$. Therefore, there exists $z_{\omega}\in Z$ such that $\omega(z_{\omega})=\widetilde{z}$.

    Let $(Y,T)$ be a minimal subsystem of $(X\times Z_{\infty},T)$ and let $\pi_{X}: Y\to X$ be the natural coordinate projection. In particular, $\pi_{X}$ is a factor map between minimal systems. Since $(x,y)\in \NRP^{[d]}(X)$, by \cref{thm: NRP_properties}, there exist $\omega_{1},\omega_{2}\in W$ such that $\left( (x,\omega_{1}),(y,\omega_{2}) \right)\in \NRP^{[d]}(Y)$. Let $z_{1}\in Z$ such that $\omega_{1}(z_{1})=\widetilde{z}$ and let $\pi: Y\to X\times Z$ with $\pi(u,\omega)=(u,\omega(z_{1}))$ for $(u,\omega)\in Y$. Set $B=\pi(Y)$, and consequently $\pi: (Y,T)\to (B,T)$ is a factor between minimal systems. Note that $\pi(x,\omega_{1})=(x,\widetilde{z})$ and $\pi(y,\omega_{2})=(y,z_{2})$ for some $z_{2}\in Z$, and \begin{align*}
        \left( (x,\widetilde{z}),(y,z_{2}) \right) = \pi \times \pi \left( (x,\omega_{1}),(y,\omega_{2}) \right) \in \NRP^{[d]}(B).
    \end{align*}

    Let $\pi_{Z}$ be the projection of $B$ onto $Z$, and hence $\pi_{Z}: (B,T)\to (Z,T)$ is a factor map. Observe that\begin{align*}
        (\widetilde{z},z_{2}) = \pi_{Z}\times \pi_{Z} \left( (x,\widetilde{z}),(y,z_{2}) \right) \in \NRP^{[d]}(Z).
    \end{align*}

    Therefore, we obtain that $\widetilde{z}=z_{2}$, so it follows that $((x,\widetilde{z}),(y,\widetilde{z})) \in B$. In particular, since $B$ is minimal, we have $N(x,U)\cap N(\widetilde{z},V)=N((x,\widetilde{z}), U\times V)$ is a syndetic set, and hence nonempty.
\end{proof}

Note that we have implicitly proved the following: if $(X,T)$ is a minimal system, with $T$ being a finitely generated nilpotent group, and $(x,y)\in \NRP^{[d]}(X)$ then $N(x,U)\cap F$ is a syndetic set for each $F \in \sF_{d}$ and each neighborhood $U$ of $y$.

To relate $\RP^{[d]}$ to the other recurrence sets defined above, we need the following lemma.

\begin{lemma}\label{thm: almost_nil_bohr}
    Let $(X,\cX,\mu,T)$ be an ergodic measure preserving system, with $T$ being a finitely generated abelian group, and let $d\geq 1$. Then for all $A\in \cX$ with $\mu(A)>0$ the set\begin{align*}
        I = \{t\in T: \mu(A\cap tA\cap \dots \cap t^{d}A)>0\}
    \end{align*}

    is an almost Nil$_{d}$ Bohr$_{0}$-set, that is, there is some subset $M$ of $T$ with zero uniform density such that $I\Delta M$ is a Nil$_{d}$ Bohr$_{0}$-set.
\end{lemma}

\begin{proof}
    By \cref{thm: correlation_nilpotent_actions}, for every $t\in T$ we have\begin{align*}
        \mu(A\cap tA\cap  \dots\cap t^{d}A)= a(t)+b(t), 
    \end{align*}

    where $a$ is a $d$-step nilsequence and $b$ tends to zero in uniform density. By {\cite[Lemma 14]{Host_Maass_nil_ordre_deux:2007}} there is a system $(Z,T)$ of order $d$, $x\in Z$ and a continuous function $\phi\in C(Z)$ such that $a(t)=\phi(tx)$.

    Suppose that $\phi(x)\leq 0$. By {\cite[Theorem 10.1]{Furstenberg_Katznelson85}} there is $\epsilon>0$ such that\begin{align*}
        \{t\in T: \mu(A\cap tA\cap  \dots\cap t^{d}A)>\epsilon\}
    \end{align*}

    is an $IP^{*}$-set. Let $V$ be a neighborhood of $x$ such that $\phi(x')<\epsilon/2$ for any $x'\in V$. Since $N(x,V)$ is an $IP^{*}$-set, we have that \begin{align*}
        N(x,V) \cap \{t\in T: \mu(A\cap tA\cap  \dots\cap t^{d}A)>\epsilon\}
    \end{align*}

    is an $IP^{*}$-set, and in particular syndetic.
    
    Since $\{t\in T: |b(t)|>\varepsilon/2\}$ has zero Banach density, there exists $t\in T$ such that $|b(t)|<\epsilon/2$, $tx\in V$ and \begin{align*}
        \mu(A\cap tA\cap  \dots\cap t^{d}A)>\epsilon.
    \end{align*}

    Therefore,\begin{align*}
       \phi(tx)+b(t) = a(t)+b(t)<\epsilon< \mu(A\cap tA\cap  \dots\cap t^{d}A),
    \end{align*}

    which is a contradiction. Hence, we conclude that $\phi(x)>0$.

    Now, for each $\varepsilon>0$, the set $\{t\in T: \phi(tx)>\phi(x)-\varepsilon/2\}$ is a Nil$_{d}$ Bohr$_{0}$-set. Finally, since $\{t\in T: |b(t)|>\varepsilon/2\}$ has zero upper Banach density, the desired result follows.
\end{proof}

\begin{theorem}\label{thm: RPd_recurrence_sets}
    Let $(X,T)$ be a minimal system, with $T$ being a finitely generated abelian group, and let $d\geq 1$ be an integer. Consider the following statements:\begin{enumerate}[label=(\arabic*)]
        \item $(x,y)\in \NRP^{[d]}(X)$.

        \item $N(x,U)\in \sF_{\operatorname{Poi}_{d}}$ for each neighborhood $U$ of $y$.
        
        \item $N(x,U)\in \sF_{\operatorname{Bir}_{d}}$ for each neighborhood $U$ of $y$.

        \item $N(x,U)\in \sF_{d}^{*}$ for each neighborhood $U$ of $y$.
    \end{enumerate}

    Then $(1)\Rightarrow (2) \Rightarrow (3)$ and $(1)\Leftrightarrow (4)$.
\end{theorem}

\begin{remark}
    We expect that these statements are actually equivalent. This is known to be true for $\Z$-actions (see \cite{Huang_Shao_Ye_nilbohr_automorphy:2016}). The proof in the $\Z$-action case relies on generalized polynomials together with quite involved computations. It would be interesting to extend their argument to actions of finitely generated abelian groups, and more generally to finitely generated nilpotent groups. Nevertheless, in this paper, we use \cref{thm: RPd_recurrence_sets} to study topological characteristic factors, and for that purpose, we only need the implications stated in \cref{thm: RPd_recurrence_sets}.
\end{remark}

\begin{proof}
    By \cref{prop: inverse_limit_metric}, we may assume that $(X,T)$ is a minimal metric system.

    First we prove $(1)\Rightarrow (2)$. Let $U$ be a neighborhood of $y$ and let $(Y,\cY,\mu,T)$ be a measure preserving system and $A\in \cY$ with $\mu(A)>0$. Let $\mu = \int_{\Omega}\mu_{\omega} \,dm(\omega)$ be an ergodic decomposition of $\mu$. Then there exists a set $\Omega'\subseteq \Omega$ with $m(\Omega')>0$ such that for every $\omega \in \Omega'$ we have $\mu_{\omega}(A)>0$. For each $\omega\in \Omega'$, define\begin{align*}
        F_{\omega} = \{t\in T: \mu_{\omega}(A\cap tA\cap \dots\cap t^{d}A)>0\}.
    \end{align*}

    By \cref{thm: almost_nil_bohr}, there exists a subset $M$ of $T$ with zero uniform density such that $F_{\omega}\Delta M$ is a Nil$_{d}$ Bohr$_{0}$-set. Consequently, $N(x,U)\cap (F_{\omega}\Delta M)$ is a syndetic set. Since M has zero uniform density, there is a $t_{\omega}\neq e$, where $e$ is the identity of $T$, with $t_{\omega} \in N(x,U)\cap F_{\omega}$. It follows that there exist a subset $\Omega'' \subseteq \Omega'$ with $m(\Omega'')>0$ and an element $t\in N(x,U)$ such that for every $\omega\in \Omega''$ one has $\mu_{\omega}(A\cap tA\cap \dots\cap t^{d}A)>0$. Therefore, $\mu(A\cap tA \cap \dots t^{d}A)>0$, and hence $N(x,U)\in \sF_{\operatorname{Poi}_{d}}$.

    Clearly $(2)\Rightarrow (3)$ and, by \cref{thm: characterization_NRP_Nil_bohr_sets}, we have $(4)\Leftrightarrow (1)$.
\end{proof}

\section{Topological characteristic factors}

\subsection{Basic definitions and properties}

In this subsection, we list the basic definitions and properties that will be used in this section.

\begin{definition}
    Let $X,Y$ be Hausdorff compact spaces  and let $\pi:X\to Y$ be a continuous surjective map. A subset $L$ of $X$ is called {\em $\pi$-saturated} if $L=\pi^{-1}(\pi(L))$.
\end{definition}

\begin{lemma}[{\cite[Lemma 3.5]{Glasner_Huang_Shao_Weiss_Ye_Topological_characteristic_factors:2020}}]\label{lemma: composition_saturated}
   Let $X,Y,Z$ be Hausdorff compact spaces. Let $\pi:X\to Y$, $\phi:X\to Z$ and $\psi: Z\to Y$ be continuous surjective maps such that $\pi=\psi\circ \phi$:
    
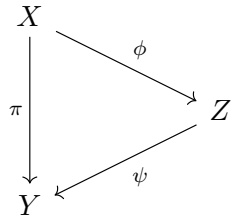
    \[\begin{tikzcd}
	X \\
	&& Z \\
	Y
	\arrow["\phi", from=1-1, to=2-3]
	\arrow["\pi"', from=1-1, to=3-1]
	\arrow["\psi", from=2-3, to=3-1]
\end{tikzcd}\]

\begin{enumerate}[label=(\arabic*)]
    \item If $A\subseteq X$ is $\pi$-saturated, then $A$ is $\phi$-saturated.
    \item If $A\subseteq X$ is $\phi$-saturated and $\phi(A)$ is $\psi$-saturated, then $A$ is $\pi$-saturated.
 \end{enumerate}
\end{lemma}

Given a topological group $T$ and an integer $d\geq 1$, set\begin{align*}
    \sigma_{d} &= \{(t,t,\dots,t): t\in T\} \subseteq T^{d},\\
    \tau_{d} &=  \{(t,t^{2},\dots,t^{d}): t\in T\}\subseteq T^{d}.
\end{align*}

For a system $(X,T)$ and an integer $d\geq 1$. Let\begin{align*}
    N_{d}(X,T)=N_{d}(X)=\overline{\tau_{d} \Delta^{(d)}(X)},
\end{align*}

where $\Delta^{(d)}(A)=\{(x,\dots,x):x\in A\}\subseteq X^{d}$ for $A\subseteq X$. 

We write $x^{(d)} = (x,\dots,x)\in X^{d}$. If $(X,T)$ is transitive and $x\in X$ a transitive point, then $N_{d}(X)= \overline{\langle\sigma_{d},\tau_{d}\rangle x^{(d)}}$. The next theorem was proved for $\Z$-actions, but its proof is also valid for group actions.

\begin{theorem}[{\cite[Proposition 1.55]{Glasner_ergodic_theory_joinings:2003}}]
    Let $(X,T)$ be a minimal system and $d\geq 1$ be an integer. Then the system $(N_{d}(X),\langle\sigma_{d},\tau_{d}\rangle)$ is minimal and the $\tau_{d}$-minimal points are dense in $N_{d}(X)$.
\end{theorem}

Moreover, by a similar proof, one obtains the following result.

\begin{theorem}\label{thm: Nd_minimal}
    Let $(X,T)$ be a minimal system and $a_{1},a_{2},\dots,a_{d}$ be distinct numbers of $\Z$, where $d\geq 1$ is an integer. Let \begin{align*}
        \tau_{a_{1},\dots,a_{d}} &=  \{(t^{a_{1}},t^{a_2},\dots,t^{a_d}): t\in T\}\\
        N_{a_{1}, \dots ,a_{d}}(X) &= \overline{\tau_{a_{1},\dots,a_{d}} \Delta^{(d)}(X)}.
    \end{align*}

    Then the system $(N_{a_{1},\dots,a_{d}}(X),\langle \sigma_{d},\tau_{a_{1},\dots,a_{d}} \rangle )$ is minimal and the $\tau_{a_{1},\dots,a_{d}}$-minimal points are dense in $N_{a_{1},\dots,a_{d}}(X)$.
\end{theorem}

\subsection{Topological characteristic factors along cubes}\label{subsec: charactersitic_cubic} In this subsection we show that the system of order $d-1$ is the topological characteristic factor along cubes of order $d$, modulo almost one-to-one factors, for any group action.

\begin{definition}
    Let $(X, T)$ be a metric system and let $\pi:(X, T)\to (Y, T)$ be a factor between metric systems. The system $(Y, T)$ is said to be a {\em topological cubic characteristic factor along cubes of order $d$} if there exists a dense $G_{\delta}$ set $X_{0}$ of $X$ such that for each $x\in X_{0}$ the set $\overline{\mathcal{F}^{[d]}_{*}(T)x^{[d]}_{*}}$ is $\pi^{[d]}_{*}$-saturated.
\end{definition}

The saturation property of dynamical cubes is closely related to the notion of a topological characteristic factor along cubes of order $d$. In particular, by \cref{thm: NRP_properties}, we immediately obtain the following result.

\begin{lemma}\label{prop: suf_nec_condition_cube_saturated}
    Let $\pi:(X,T)\to (Y,T)$ be a factor between minimal metric systems and let $d\geq 1$ be an integer. If $\bQ^{[d]}(X)$ is $\pi^{[d]}$-saturated, then $(Y,T)$ is a topological characteristic factor along cubes of order $d$.
\end{lemma}

Now, we prove the main result of this subsection.

\begin{proposition}\label{prop: cube_saturated}
    Let $\pi: (X, T)\to (Y, T)$ be a factor of minimal systems and $d\geq 2$ be an integer. If $\pi$ is open and $X_{d-1}$ is a factor of $Y$, where $X_{d-1}$ is the maximal factor of order $d-1$ of $X$, then $\bQ^{[d]}(X)$ is $\pi^{[d]}$-saturated.
\end{proposition}

\begin{proof}
    Let $\bx \in \bQ^{[d]}(X)$. Assume first that $\bx$ is a $\cF^{[d]}$-minimal point. In particular, we have $\bx \in \overline{\cF^{[d]}\bx_{\emptyset}^{[d]}}$ since $\overline{\cF^{[d]}\bx_{\emptyset}^{[d]}}$ is the unique $\cF^{[d]}$-minimal system in $\bQ^{[d]}[\bx_{\emptyset}]$, where $\bQ^{[d]}[\bx_{\emptyset}] = \{\by: \by_{\emptyset}=x_{\emptyset}\}$.
    
    By the invariance of $\bQ^{[d]}(X)$ under Euclidean permutations, to prove that $(\pi^{[d]})^{-1}(\pi^{[d]}(\bx)) \subseteq \bQ^{[d]}(X)$, it is enough to show that $(y,\bx_{*})\in \bQ^{[d]}(X)$ whenever $(\bx_{\emptyset},y)\in R_{\pi} \subseteq \NRP^{[d-1]}(X)$. By \cref{thm: cube_is_minimal}, we have $\overline{\cF^{[d]}_{*}y^{[d]}_{*}} = \overline{\cF^{[d]}_{*}\bx_{\emptyset *}^{[d]}}$. Therefore, \begin{align*}
        (y,\bx_{*}) \in \overline{\cF^{[d]}y^{[d]}} \subseteq \bQ^{[d]}(X).
    \end{align*}

    Thus, it follows that $(\pi^{[d]})^{-1}(\pi^{[d]}(\bx)) \subseteq \bQ^{[d]}(X)$.

    Now, we will prove the general case, where $\bx$ is not necessarily a $\cF^{[d]}$-minimal point. By \cref{thm: cube_is_minimal}, there exists a net $(\bx_{i})_{i\in I}$ in $\bQ^{[d]}(X)$ of $\cF^{[d]}$-minimal points such that $\bx_{i}\to \bx$. Since $\pi$ is an open map, it follows that\begin{align*}
        (\pi^{[d]})^{-1}(\pi^{[d]}(\bx_{i})) \to (\pi^{[d]})^{-1}(\pi^{[d]}(\bx)).
    \end{align*}

    Since $\bx_{i}$ is a $\cF^{[d]}$-minimal point, we have $(\pi^{[d]})^{-1}(\pi^{[d]}(\bx_{i})) \subseteq \bQ^{[d]}(X)$. Hence, we conclude that $(\pi^{[d]})^{-1}(\pi^{[d]}(\bx)) \subseteq \bQ^{[d]}(X)$.
\end{proof}

By \cref{thm: AG_diagram} and \cref{prop: cube_saturated}, we obtain the following theorem.

\begin{theorem}
        Let $(X, T)$ be a minimal metric system and $d\geq 2$ be an integer. Let $\pi:(X, T)\to (X_{d-1}, T)$ be the factor map to the maximal factor of order $d-1$. Then, there is a commutative diagram of factors of minimal systems

        \[\begin{tikzcd}
	X && {X'} \\
	\\
	{X_{d-1}} && {X_{d-1}'}
	\arrow["\pi"', from=1-1, to=3-1]
	\arrow["{\theta'}"', from=1-3, to=1-1]
	\arrow["{\pi'}", from=1-3, to=3-3]
	\arrow["\theta", from=3-3, to=3-1]
    \end{tikzcd}\]

    such that $\bQ^{[d]}(X') = ({\pi'}^{[d]})^{-1}(\bQ^{[d]}(X_{d-1}'))$, where $\theta$ and $\theta'$ are almost one-to-one extensions.
\end{theorem}

Thus, by \cref{lemma: composition_saturated}, we conclude the following theorem.

\begin{theorem}
    Let $(X, T)$ be a minimal metric system and $d\geq 2$ be an integer. Let $\pi:(X, T)\to (X_{d-1}, T)$ be the factor to the maximal factor of order $d-1$. Then, there is a commutative diagram of factors of minimal systems

    \[\begin{tikzcd}
	X && {X'} \\
	\\
	{X_{d-1}} && {X_{d-1}'}
	\arrow["\pi"', from=1-1, to=3-1]
	\arrow["{\theta'}"', from=1-3, to=1-1]
	\arrow["{\pi'}", from=1-3, to=3-3]
	\arrow["\theta", from=3-3, to=3-1]
    \end{tikzcd}\]

    such that $(X_{d-1}',T)$ is the topological characteristic factor along cubes of order $d$ of $(X',T)$, where $\theta,\theta'$ are almost one-to-one factors.
\end{theorem}

Moreover, by \cref{lemma: composition_saturated}, we conclude that the maximal factor of order $\infty$ is a topological characteristic factor along cubes of every order, modulo almost one-to-one factors.

\begin{theorem}\label{thm: cubic_characteristic_infty}
    Let $(X, T)$ be a minimal metric system and $d\geq 2$ be an integer. Let $\pi:(X, T)\to (X_{\infty}, T)$ be the factor to the maximal factor of order $\infty$. Then, there is a commutative diagram of factors of minimal systems

        \[\begin{tikzcd}
	X && {X'} \\
	\\
	{X_{\infty}} && {X_{\infty}'}
	\arrow["\pi"', from=1-1, to=3-1]
	\arrow["{\theta'}"', from=1-3, to=1-1]
	\arrow["{\pi'}", from=1-3, to=3-3]
	\arrow["\theta", from=3-3, to=3-1]
    \end{tikzcd}\]

    such that $(X_{\infty}',T)$ is the topological characteristic factor along cubes of order $d$ of $(X',T)$, where $\theta,\theta'$ are almost one-to-one factors.
\end{theorem}

\subsection{Topological characteristic factor along arithmetic progressions}\label{subsec: characteristic_arithmetic}

In this subsection we show that the system of order $d-1$ is the topological characteristic factor along arithmetic progressions of order $d$ (or the topological characteristic factor of order $d$) for finitely generated abelian actions. Throughout, we assume that in every system $(X,T)$, the acting group $T$ is a finitely generated abelian group.

\begin{definition}
    Let $(X, T)$ be a metric system and let $\pi:(X, T)\to (Y, T)$ be a factor. The system $(Y, T)$ is said to be a {\em topological characteristic factor of order $d$} if there exists a dense $G_{\delta}$ set $X_{0}$ of $X$ such that for each $x\in X_{0}$ the set $\overline{\tau_{d}x^{(d)}}$ is $\pi^{(d)}$-saturated.
\end{definition}

\begin{lemma}\label{lemma: RPd_finite_index}
    Let $(X,T)$ be a system and let $d\geq 1$ be an integer. If $H$ is a subgroup of finite index of $T$, then $\RP^{[d]}(X,T)=\RP^{[d]}(X,H)$.
\end{lemma}

\begin{proof}
    Clearly $\RP^{[d]}(X,H)\subseteq \RP^{[d]}(X,T)$. Conversely, let $(x,y)\in \RP^{[d]}(X,T)$. Let $\alpha\in \mathcal{U}_{X}$ and let $\boldsymbol{\alpha}\in \mathcal{U}_{X_{*}^{[d]}}$.

    Note that $\mathcal{F}_{*}^{[d]}(H)$ is subgroup of finite index of $\mathcal{F}_{*}^{[d]}(T)$. Then there exist $f_{1},\dots,f_{n} \in \mathcal{F}_{*}^{[d]}(T)$ such that \begin{align*}
        \mathcal{F}_{*}^{[d]}(T) = \bigcup_{i=1}^{n}f_{i}\mathcal{F}_{*}^{[d]}(H).
    \end{align*}

    Consider $\boldsymbol{\beta}\in \mathcal{U}_{X^{[d]}_{*}}$ such that $(f_{i}^{-1}\boldsymbol{a},f_{i}^{-1}\boldsymbol{b})\in \boldsymbol{\alpha}$ whenever $(\boldsymbol{a},\boldsymbol{b})\in \boldsymbol{\beta}$, for every $i=1,\dots,n$. 
    
    Since $(x,y)\in \RP^{[d]}(X)$, there exist $x',y'\in X$ and $f\in \cF^{[d]}_{*}(T)$ such that $(x,x'),(y,y')\in \alpha$ and $(f{x'}^{[d]}_{*},f{y'}^{[d]}_{*})\in \boldsymbol{\beta}$. Note that $f=f_{i}f'$ for some $f'\in \mathcal{F}^{[d]}(H)$ and for some $i\in \{1,\dots,n\}$. Therefore, we have $(f'{x'}^{[d]}_{*},f'{y'}^{[d]}_{*})\in \boldsymbol{\alpha}$, and hence $(x,y)\in \RP^{[d]}(X,H)$.
\end{proof}

Now, we show that $N_{d}$ is saturated with respect to the quotient map to  $\RP^{[\infty]}$, modulo almost one-to-one factor. To prove this, we use the key idea of the proof of {\cite[Lemma 4.1]{Glasner_Huang_Shao_Weiss_Ye_Topological_characteristic_factors:2020}}.

\begin{proposition}\label{prop: Nd_saturated_infty_nil}
    Let $\pi: (X, T)\to (Y, T)$ be a factor of minimal systems and $d\geq 2$ be an integer. If $\pi$ is open and $X_{\infty}$ is a factor of $Y$, where $X_{\infty}$ is the maximal factor of order $\infty$ of $X$, then $N_{d}(X)$ is $\pi^{(d)}$-saturated.
\end{proposition}

\begin{proof}
    Let $1\leq j \leq d$ and define\begin{align*}
        S_{j} = \{(t^{j-1},t^{j-2},\dots,t,e,t^{-1},\dots,t^{-(d-j)}):t\in T\},
    \end{align*}

    where $e$ is the identity of $T$. By \cref{thm: Nd_minimal}, the system $(N_{d}(X),\langle \sigma_{d},S_{j}\rangle)$ is minimal, and the set of $S_{j}$-minimal points of $N_{d}(X)$ is dense in $N_{d}(X)$.

    Let $\bx = (x_{1},x_{2},\dots,x_{d})\in N_{d}(X)$ be an $S_{j}$-minimal point and let $y\in \pi^{-1}(\pi(x_{j}))$. Set $Z= \overline{S_{j}\bx}$ and let \begin{align*}
        U = Z \cap U_{1}\times U_{2}\times \dots\times U_{d},
    \end{align*}

    where $U_{1}\times U_{2}\times \dots \times U_{d}$ is a neighborhood of $\bx$.

    Since $X_{\infty}$ is a factor of $Y$, it follows that $(x_{j},y)\in \RP^{[d']}(X)$ with $d'=2d!(d-1)!$. By \cref{lemma: RPd_finite_index}, we have that $(x_{j},y)\in \RP^{[d']}(X,T^{j(d-1)!})$. Note that \cref{thm: RPd_recurrence_sets} holds for systems consisting of finitely many minimal subsystems. In particular, since $T^{j(d-1)!}$ is a subgroup of finite index, \cref{thm: RPd_recurrence_sets} holds for $(X,T^{j(d-1)!})$. Therefore, $N_{T^{j(d-1)!}}(x_{j},V)\in \sF_{\operatorname{Bir}_{d'}}$ for each neighborhood $V$ of $y$. Since $(Z,S_{j})$ is a minimal system, there exists $t\in T$ such that\begin{align*}
        t^{j(d-1)!}x_{j}\in V,\\
        U \cap \boldsymbol{s}_{j}U\cap \boldsymbol{s}_{j}^{2}U\cap \dots\cap \boldsymbol{s}_{j}^{d'}U\neq \emptyset,
    \end{align*}

    where $\boldsymbol{s}_{j} = (t^{j-1},t^{j-2},\dots,t,e,t^{-1},\dots,t^{-(d-j)})$. 
    
    Since $d'=2d!(d-1)!$, there exists $\bz \in Z$ such that\begin{align*}
        \boldsymbol{s}\bz \in U_{1}\times U_{2}\times \dots\times U_{j-1}\times V\times U_{j+1} \dots U_{d},
    \end{align*}

    where $\boldsymbol{s} = (t^{(d-1)!},t^{2(d-1)!},\dots, t^{j(d-1)!},\dots,t^{d(d-1)!})$. 
    
    Thus, there exist a net $(\bz_{i})_{i\in I}$ in $Z$ and a net $(t_{i})_{i\in I}$ in $T$ such that \begin{align*}
        \boldsymbol{s}_{i}\bz_{i} \to (x_{1},x_{2},\dots,x_{j-1},y,x_{j+1},\dots,x_{d}),
    \end{align*}

    where $\boldsymbol{s}_{i} = (t_{i}^{(d-1)!},t_{i}^{2(d-1)!},\dots, t_{i}^{j(d-1)!},\dots,t_{i}^{d(d-1)!})$. 
    
    Since $Z\subseteq N_{d}(X)$, we obtain \begin{align*}
        (x_{1},x_{2},\dots,x_{j-1},y,x_{j+1},\dots,x_{d}) \in N_{d}(X).
    \end{align*}

    Now, we will prove the general case, where $\bx$ is not necessarily a $S_{j}$-minimal point. Let $\bx = (x_{1},\dots,x_{d})\in N_{d}(X)$ and let $y\in \pi^{-1}(\pi(x_{j}))$. By \cref{thm: Nd_minimal}, there exists a net $(\bx_{i})_{i\in I}$ in $N_{d}(X)$ of $S_{j}$-minimal points $\bx_{i}=(\bx_{i,1},\dots,\bx_{i,d})\to \bx$. By the definition of $S_{j}$, we have $\bx_{i,j}=x_{j}$. Since $\pi$ is an open map, it follows that\begin{align*}
        \pi^{-1}(\pi(\bx_{i,j}))\to \pi^{-1}(\pi(x_{j})).
    \end{align*}

    Therefore, there exists a net $(y_{i})_{i\in I}$ in $\pi^{-1}(\pi(\bx_{i,j}))$ with $y_{i}\to y$. Since $\bx_{i}$ is a $S_{j}$-minimal point and $\pi(\bx_{i,j})=\pi(y_{i})$, we have that \begin{align*}
        (\bx_{i,1},\bx_{i,2},\dots,\bx_{i,j-1},y_{i},\bx_{i,j+1},\dots,\bx_{i,d}) \in N_{d}(X).
    \end{align*}

    Since $N_{d}(X)$ is closed, we obtain \begin{align*}
        (x_{1},x_{2},\dots,x_{j-1},y,x_{j+1},\dots,x_{d}) \in N_{d}(X).
    \end{align*}

    Hence, we conclude that $(\pi^{(d)})^{-1}(\pi^{(d)}(\bx))\subseteq N_{d}(X)$.
\end{proof}

Now, to pass from the factor of order $\infty$ to the factor of order $d-1$, we follow the ideas developed in \cite{Ye_Yu_polynomial_saturation:2025}. Since the proof for the distal case is essentially the same as that of {\cite[Theorem 3.6]{Ye_Yu_polynomial_saturation:2025}}, we omit it.

\begin{proposition}\label{prop: Nd_saturated_d_nil}
    Let $(X,T)$ be a minimal distal metric system and $d\geq 2$ be an integer. Then $N_{d+1}(X)$ is $\pi^{(d+1)}_{d-1}$-saturated, where $\pi_{d-1}:X \to X_{d-1}$ is the factor map from $X$ to its maximal factor of order $d-1$.
\end{proposition}

Finally we can show the general case using the following lemmas.

\begin{lemma}[{\cite[Proposition 3.1]{Veech_point-distal_flows:1970}}]\label{lemma: density_preimage}
    Let $\pi:(X,T)\to(Y,T)$ be a factor between metric minimal systems. If $X_{0}$ is a dense $G_{\delta}$ set of $X$, then there is a dense $G_{\delta}$ set $Y_{0}$ of $Y$ such that for each $y\in Y_{0}$, $X_{0}\cap \pi^{-1}(y)$ is a dense $G_{\delta}$ set of $\pi^{-1}(y)$.
\end{lemma}

\begin{lemma}\label{lemma: density_continuty_points}
    Let $\pi:(X,T)\to(Y,T)$ be a factor between metric minimal systems and let $\Phi: Y \to 2^{X}$ defined by $y\mapsto \pi^{-1}(y)$. Then the set of continuity points of $\Phi$ is a dense $G_{\delta}$ set of $Y$.
\end{lemma}

\begin{theorem}\label{thm: topological_dstep_char_factor}
    Let $(X, T)$ be a minimal metric system and $d\geq 2$ be an integer. Let $\pi:(X, T)\to (X_{d-1}, T)$ be the factor to the maximal factor of order $d-1$. Then, there is a commutative diagram of factors of minimal systems

    \[\begin{tikzcd}
	X && {X'} \\
	\\
	{X_{d-1}} && {X_{d-1}'}
	\arrow["\pi"', from=1-1, to=3-1]
	\arrow["{\theta'}"', from=1-3, to=1-1]
	\arrow["{\pi'}", from=1-3, to=3-3]
	\arrow["\theta", from=3-3, to=3-1]
    \end{tikzcd}\]

    such that $N_{d+1}(X')$ is ${\pi'}^{(d+1)}$-saturated, where $\theta,\theta'$ are almost one-to-one factors.
\end{theorem}

\begin{proof}
    Let $\pi_{\infty}:X\to X_{\infty}$ and $\pi_{d-1}:X_{\infty}\to X_{d-1}$ be the natural factor maps. Note that $\pi = \pi_{d-1}\circ\pi_{\infty}$.
    
    We first prove that there exists a dense $G_{\delta}$ set $X_{0}$ of $X$ such that $(\pi^{(d+1)}_{\infty})^{-1}(\pi^{(d+1)}_{\infty}(\bt x^{(d+1)})) \subseteq N_{d+1}(X)$ for every $x\in X_{0}$ and every $\bt\in \tau_{d}$. By \cref{prop: Nd_saturated_infty_nil}, there is a commutative diagram of factors of minimal systems

    \[\begin{tikzcd}
	X && {X'} \\
	\\
	{X_{\infty}} && {X_{\infty}'}
	\arrow["\pi_{\infty}"', from=1-1, to=3-1]
	\arrow["{\theta'}"', from=1-3, to=1-1]
	\arrow["{\pi_{\infty}'}", from=1-3, to=3-3]
	\arrow["\theta", from=3-3, to=3-1]
    \end{tikzcd}\]

    such that $N_{d+1}(X')$ is ${\pi_{\infty}'}^{(d+1)}$-saturated, where $\theta,\theta'$ are almost one-to-one factors.

    Define dense $G_{\delta}$ sets $X_{1} = \{y\in X_{\infty}: |\theta^{-1}(y)|=1\}$ and $X_{2} = \{x\in X: |{\theta'}^{-1}(x)|=1\}$. Let $X_{0} = \pi_{\infty}^{-1}(X_{1}) \cap X_{2}$, which is a dense $G_{\delta}$ set of $X$. 
    
    Let $x\in X_{0}$, $\bt \in \tau_{d}$ and $\bx \in (\pi_{\infty}^{(d+1)})^{-1}(\pi_{\infty}^{(d+1)}(\bt x^{(d+1)}))$. Let $y\in X'$ be such that ${\theta'}^{-1}(x) = \{y\}$. Since $\pi_{\infty}(x)\in X_{1}$, we have $\theta^{-1}(\pi_{\infty}(x)) = \{\pi_{\infty}'(y)\}$. It follows from $\theta\pi_{\infty}' = \pi_{\infty}\theta'$ that there exists $\by \in ({\pi_{\infty}'}^{(d+1)})^{-1}({\pi_{\infty}'}^{(d+1)}(\bt y^{(d+1)}))$ such that ${\theta'}^{(d+1)}(\by)=\bx$. Since $N_{d+1}(X')$ is ${\pi_{\infty}'}^{(d+1)}$-saturated, we have $\by \in N_{d+1}(X')$, and consequently we obtain that $\bx \in N_{d+1}(X)$.

    Now, we show that there exists a dense $G_{\delta}$ set $\Omega$ of $X$ such that $(\pi^{(d+1)})^{-1}(\pi^{(d+1)}(x^{(d+1)})) \subseteq N_{d+1}(X)$ for each $x\in \Omega$.
    
    By \cref{lemma: density_continuty_points}, there is a dense $G_{\delta}$ set $\Omega_{1}$ of $X_{\infty}$ consisting of continuity points of the map $y\mapsto \pi^{-1}_{\infty}(y)$. Recall that $X_{0}$ is a dense $G_{\delta}$ set of $X$ such that $(\pi^{(d+1)}_{\infty})^{-1}(\pi^{(d+1)}_{\infty}(\bt x^{(d+1)})) \subseteq N_{d+1}(X)$ for every $x\in X_{0}$ and every $\bt\in \tau_{d}$. Since $X_{0} \cap \pi_{\infty}^{-1}(\Omega_{1})$ is a dense $G_{\delta}$ set of $X$, by \cref{lemma: density_preimage}, there exists a dense $G_{\delta}$ set $Y_{0}$ of $X_{d-1}$ such that for each $y\in Y_{0}$, $X_{0} \cap \pi_{\infty}^{-1}(\Omega_{1})\cap \pi^{-1}(y)$ is a dense $G_{\delta}$ set of $\pi^{-1}(y)$. Define $\Omega = X_{0} \cap \pi_{\infty}^{-1}(\Omega_{1}) \cap \pi^{-1}(Y_{0})$.

    Let $x\in \Omega$. Since $x\in \pi^{-1}(Y_{0})$, it is enough to prove that $\bx \in N_{d+1}(X)$ for each $\bx \in (X_{0} \cap \pi_{\infty}^{-1}(\Omega_{1}))^{(d+1)} \cap (\pi^{(d+1)})^{-1}(\pi^{(d+1)}(x^{(d+1)}))$. Let $\bx \in (X_{0} \cap \pi_{\infty}^{-1}(\Omega_{1}))^{(d+1)} \cap (\pi^{(d+1)})^{-1}(\pi^{(d+1)}(x^{(d+1)}))$. By \cref{prop: Nd_saturated_d_nil}, it follows that\begin{align*}
        \bx \in (\pi_{\infty}^{(d+1)})^{-1}((\pi_{d-1}^{(d+1)})^{-1}(\pi_{d-1}^{(d+1)}(\pi_{\infty}^{(d+1)}(x^{(d+1)})))) \subseteq (\pi_{\infty}^{(d+1)})^{-1}(N_{d+1}(X_{\infty})).
    \end{align*}

    By density, there exist a sequence $(\bt_{i})_{i\in \N}$ in $\tau_{d}$ and a sequence $(x_{i})_{i\in\N}$ in $X_{0}$ such that \begin{align*}
        \pi_{\infty}^{(d+1)}(\bt_{i}x_{i}^{(d+1)})\to \pi_{\infty}^{(d+1)}(\bx).
    \end{align*}

    Since $\bx \in \pi_{\infty}^{(d+1)}(\Omega_{1}^{(d+1)})$, we get \begin{align*}
        (\pi_{\infty}^{(d+1)})^{-1}(\pi_{\infty}^{(d+1)}(\bt_{i}x_{i}^{(d+1)}))\to (\pi_{\infty}^{(d+1)})^{-1}(\pi_{\infty}^{(d+1)}(\bx)).
    \end{align*}

    Since $(x_{i})_{i\in \N} \subseteq X_{0}$, for each $i\in \N$, it follows that \begin{align*}
        (\pi_{\infty}^{(d+1)})^{-1}(\pi_{\infty}^{(d+1)}(\bt_{i}x_{i}^{(d+1)})) \subseteq N_{d+1}(X).
    \end{align*}

    Hence, $\bx \in N_{d+1}(X)$.

    Finally, by \cref{thm: AG_diagram}, there is a commutative diagram of factors of minimal systems

    \[\begin{tikzcd}
	X && {X'} \\
	\\
	{X_{d-1}} && {X_{d-1}'}
	\arrow["\pi"', from=1-1, to=3-1]
	\arrow["{\theta'}"', from=1-3, to=1-1]
	\arrow["{\pi'}", from=1-3, to=3-3]
	\arrow["\theta", from=3-3, to=3-1]
    \end{tikzcd}\]

    where $\theta,\theta'$ are almost one-to-one factors. Since $\pi'$ is open, to show that $N_{d+1}(X')$ is ${\pi'}^{(d+1)}$-saturated, it is enough to prove that $({\pi'}^{(d+1)})^{-1}({\pi'}^{(d+1)} (x^{(d+1)}))) \subseteq N_{d+1}(X')$ for each $x\in X'$.

    Since $\theta$ is an almost one-to-one factor and by the result proved above, there exists a dense $G_{\delta}$ set $\Omega'$ of $X'$ such that $|\theta^{-1}(\theta(\pi'(x)))| = 1$ and \begin{align*}
        ({\pi'}^{(d+1)})^{-1} ((\theta^{(d+1)})^{-1} ( \theta^{(d+1)}( ({\pi'}^{(d+1)} (x^{(d+1)}))) ))\subseteq N_{d+1}(X')
    \end{align*} 

    for each $x\in \Omega'$. 
    
    Since $|\theta^{-1}(\theta(\pi'(x)))| = 1$, we have $({\pi'}^{(d+1)})^{-1}({\pi'}^{(d+1)} (x^{(d+1)}))) \subseteq N_{d+1}(X')$ for each $x\in \Omega'$. Now, let $x\in X'$ and consider $(x_{i})_{i\in \N} \subseteq \Omega'$ such that $x_{i}\to x$. Since $\pi'$ is open, it follows that \begin{align*}
        ({\pi'}^{(d+1)})^{-1}({\pi'}^{(d+1)} (x_{i}^{(d+1}))) \to ({\pi'}^{(d+1)})^{-1}({\pi'}^{(d+1)} (x^{(d+1)}))).
    \end{align*}

    Since $({\pi'}^{(d+1)})^{-1}({\pi'}^{(d+1)} (x_{i}^{(d+1}))) \subseteq N_{d+1}(X')$, we conclude that $({\pi'}^{(d+1)})^{-1}({\pi'}^{(d+1)} (x^{(d+1)})))\subseteq N_{d+1}(X')$.
 \end{proof}

 The following was implicitly proved in \cref{thm: topological_dstep_char_factor}, but we state it here for convenience for the subsequent subsection.

\begin{lemma}\label{lemma: dense_preimage_in_Nd_infty}
     Let $(X, T)$ be a minimal metric system. Let $\pi:(X, T)\to (X_{\infty}, T)$ be the factor to the maximal factor of order $\infty$. Then there exists a dense $G_{\delta}$ set $X_{0}$ of $X$ such that, for every integer $d\geq 1$ and every $x\in X_{0}$,  $(\pi^{(d)})^{-1}(\pi^{(d)}(x^{(d)})) \subseteq N_{d}(X)$.
 \end{lemma}

Moreover, using the same proof as in \cref{thm: topological_dstep_char_factor}, we can deduce the following lemma.

 \begin{lemma}\label{lemma: dense_preimage_in_Nd}
     Let $(X, T)$ be a minimal metric system and $d\geq 2$ be an integer. Let $\pi:(X, T)\to (X_{d-1}, T)$ be the factor to the maximal factor of order $d-1$. Then there exists a dense $G_{\delta}$ set $X_{0}$ of $X$ such that $(\pi^{(d+1)})^{-1}(\pi^{(d+1)}(x^{(d+1)})) \subseteq N_{d+1}(X)$ for every $x\in X_{0}$.
 \end{lemma}

By the following lemma, which was proved for $\mathbb{Z}$-actions and whose proof also works for general group actions, we deduce that the factor of order $d-1$ is the topological characteristic factor of order $d$.

 \begin{lemma}[{\cite[Theorem 4.5]{Wu_Xu_Ye_Structure_saturated:2023}}]
    Let $\pi:(X,T)\to (Y,T)$ be a factor between minimal metric systems and let $d\geq 1$ be an integer. If $N_{d+1}(X)$ is $\pi^{(d+1)}$-saturated, then $Y$ is a $d$-step topological characteristic factor of $X$.
\end{lemma}

\begin{theorem}
    Let $(X, T)$ be a minimal metric system and $d\geq 2$ be an integer. Let $\pi:(X, T)\to (X_{d-1}, T)$ be the factor to the maximal factor of order $d-1$. Then, there is a commutative diagram of factors of minimal systems

    \[\begin{tikzcd}
	X && {X'} \\
	\\
	{X_{d-1}} && {X_{d-1}'}
	\arrow["\pi"', from=1-1, to=3-1]
	\arrow["{\theta'}"', from=1-3, to=1-1]
	\arrow["{\pi'}", from=1-3, to=3-3]
	\arrow["\theta", from=3-3, to=3-1]
    \end{tikzcd}\]

    such that $(X_{d-1}',T)$ is the topological characteristic factor of order $d$ of $(X',T)$, where $\theta,\theta'$ are almost one-to-one factors.
\end{theorem}
 \subsection{Applications}

 First, we note an application of \cref{thm: cubic_characteristic_infty}. In {\cite[Theorem 5.9]{Alvarez_Donoso_cube_struct_univ_nil_applications:2025}}, we proved that the maximal distal factor of a dynamical cube is the dynamical cube of the maximal distal factor. The main theorem used in that proof is {\cite[Proposition 5.16]{Alvarez_Donoso_cube_struct_univ_nil_applications:2025}}. By \cref{thm: cubic_characteristic_infty}, the same result holds for any group action. We omit the proof, since it is exactly the same.

 \begin{theorem}
     Let $(X,T)$ be a minimal system and let $d\geq 1$ be an integer. Then, $\bQ^{[d]}(X_{dis})$ is the maximal distal factor of $\bQ^{[d]}(X)$, where $X_{dis}$ is the maximal distal factor of $X$.
 \end{theorem}

We now give some applications of the topological characteristic factor along arithmetic progressions. Throughout, we assume that in every system $(X,T)$, the acting group $T$ is a finitely generated abelian group. First, we need the following definition, which was studied in \cite{Glasner_Huang_Shao_Ye_regionally_arithmetic_prog_nil:2020} for $\Z$-actions.

 \begin{definition}
     Let $(X,T)$ be a system and let $d\geq 1$ be an integer. We say that $(x,y)\in X\times X$ is a {\em regionally proximal pair of order $d$ along arithmetic progressions} if there are nets $(t_{i})_{i\in I} \subseteq T$, $((x_{i},y_{i}))_{i\in I}\subseteq X\times X$ and $a_{*}\in X^{d}$ such that\begin{align*}
         (x_{i},t_{i}x_{i},t_{i}^{2}x_{i},\dots,t_{i}^{d}x_{i},y_{i},t_{i}y_{i},t_{i}^{2}y_{i}\dots,t_{i}^{d}y_{i})\to (x,a_{*},y,a_{*}).
     \end{align*}

     The set of all such pairs is denoted by $\AP^{[d]}(X)$ and is called the {\em regionally proximal relation of order $d$ along arithmetic progressions}.
 \end{definition}

We next give a characterization of the minimal points of $\RP^{[d]}$.

 \begin{theorem}\label{prop: characterization_minimal_in_RPd}
     Let $(X,T)$ be a minimal system and let $d\geq 1$ be an integer. If $(x,y)\in X\times X$ is a minimal point, then the following statements are equivalent:\begin{enumerate}[label=(\arabic*)]
         \item $(x,y)\in \RP^{[d]}(X)$.
         \item $\{x,y\}^{d+2} \subseteq N_{d+2}(X)$.
         \item $(x,y)\in \AP^{[d]}(X)$.
     \end{enumerate}
 \end{theorem}

 \begin{proof}
    Clearly $(2)\Rightarrow (3) \Rightarrow (1)$. Thus, it suffices to show that $(1)\Rightarrow (2)$. By \cref{prop: inverse_limit_metric}, {\cite[Proposition A.2.3]{Huang_Shao_Ye_nilbohr_automorphy:2016}} and {\cite[Lemma 3.6]{Lian_Qiu_pronil_arith_progressions:2024}}, we may assume that $(X,T)$ is a metric system.
    
     By \cref{lemma: dense_preimage_in_Nd}, there exists a dense $G_{\delta}$ set $X_{0}$ of $X$ such that $(\pi^{(d+2)})^{-1}(\pi^{(d+2)}(z^{(d+2)})) \subseteq N_{d+2}(X)$ for every $z\in X_{0}$, where $\pi:X\to X_{d}$ is the factor map onto the maximal factor of order $d$. Let $x'\in X_{0}$ and let $p\in M$ be such that $x=px'$. Since $\RP^{[d]}(X)$ is invariant, we have $(x',p^{-1}y)\in \RP^{[d]}(X)$. Therefore, \begin{align*}
         \{x',p^{-1}y\}^{d+2} \subseteq (\pi^{(d+2)})^{-1}(\pi^{(d+2)}({x'}^{(d+2)})) \subseteq N_{d+2}(X).
     \end{align*}

     Since $N_{d+2}(X)$ is $T$-invariant, it follows that $\{x,y\}^{d+2} \subseteq N_{d+2}(X)$.
 \end{proof}

Note that using the same proof it is possible to obtain the following result.\begin{corollary}\label{cor: APd_RPd_open_map_nil}
    Let $(X,T)$ be a minimal system and let $d\geq 1$ be an integer. Suppose that the factor map from $X$ to its maximal factor of order $d$ is open, then the following are equivalent:\begin{enumerate}[label=(\roman*)]
         \item $(x,y)\in \RP^{[d]}(X)$.
         \item $\{x,y\}^{d+2} \subseteq N_{d+2}(X)$.
         \item $(x,y)\in \AP^{[d]}(X)$.
     \end{enumerate}
\end{corollary}

In particular \cref{cor: APd_RPd_open_map_nil} proves that $\AP^{[d]}(X)=\RP^{[d]}(X)$ for distal systems.

One consequence of \cref{prop: characterization_minimal_in_RPd} is that $\RP^{[d]}$ is the smallest closed invariant equivalence relation containing $\AP^{[d]}$. To prove this, we need the following lemma, which was proved for $\Z$-actions, but the same proof works for finitely generated abelian actions.

\begin{lemma}[{\cite[Theorem 3.8]{Glasner_Huang_Shao_Ye_regionally_arithmetic_prog_nil:2020}}]\label{lemma: Rpi_APd_proximal_factor}
    Let $\pi:(X,T)\to (Y,T)$ be a proximal factor between minimal systems and let $d\geq 1$ be an integer. Then $R_{\pi} \subseteq \AP^{[d]}(X)$.
\end{lemma}

\begin{theorem}
    Let $(X,T)$ be a minimal system and let $d\geq 1$ be an integer. Then $\RP^{[d]}(X)$ is the smallest closed invariant equivalence relation containing $\AP^{[d]}$.
\end{theorem}

\begin{proof}
    Let $\mathcal{A}(\AP^{[d]}(X))$ denote the smallest closed invariant equivalence relation containing $\AP^{[d]}$. Clearly $\mathcal{A}(\AP^{[d]}(X)) \subseteq \RP^{[d]}(X)$.

    Let $(x,y)\in \RP^{[d]}(X)$. By \cref{prop: characterization_minimal_in_RPd}, it follows that $(y,vx)\in \AP^{[d]}(X)$, where $v$ is a minimal idempotent in $M$ such that $vy=y$. Furthermore, by \cref{prop: characterization_minimal_in_RPd} and \cref{lemma: Rpi_APd_proximal_factor}, we deduce that $\P(X)\subseteq \mathcal{A}(\AP^{[d]}(X))$. In particular, $(x,vx)\in \mathcal{A}(\AP^{[d]}(X))$. Since $\mathcal{A}(\AP^{[d]}(X))$ is an equivalence relation, we conclude that $(x,y)\in \mathcal{A}(\AP^{[d]}(X))$.
\end{proof}

Now, we  use the characterization of the minimal points of $\RP^{[d]}$ to study independence. 

\begin{definition}
    Let $(X,T)$ be a system. Given a tuple $\mathscr{A} = (U_{0},U_{1},\dots,U_{k})$ of subsets of $X$, we say that $F\subseteq T$ is an independence set for $\mathscr{A}$ if for any non-empty finite set $J\subseteq T$ and any $s = (s(j):j\in J) \in \{0,1,\dots,k\}^{J}$ we have\begin{align*}
        \bigcap_{j\in J} j U_{s(j)} \neq \emptyset.
    \end{align*}

    We shall denote the collection of all independence sets for $\mathscr{A}$ by $\operatorname{Ind}(U_{0},U_{1},\dots,U_{k})$.
\end{definition}

\begin{definition}
    Let $(X,T)$ be a system and let $d\geq 1$ be an integer. A pair $(x_{1},x_{2})\in X\times X$ is called an $\operatorname{Ind}_{ap}$-pair of order $d$ if for every pair of neighborhoods $U_{1},U_{2}$ of $x_{1}$ and $x_{2}$ respectively, there exists some $t\in T\setminus\{e\}$, where $e$ is the identity of $T$, such that $\{t,t^{2},\dots,t^{d}\} \in \operatorname{Ind}(U_{1},U_{2})$.

    A pair $(x_{1},x_{2})\in X\times X$ is called an $\operatorname{Ind}_{ap}$-pair if for every integer $d\geq 1$, $(x_{1},x_{2})$ is an $\operatorname{Ind}_{ap}$-pair of order $d$.
\end{definition}

Denote by $\operatorname{Ind}_{ap}(X)$ the set of all $\operatorname{Ind}_{ap}$-pairs of order $d$. Note that $\operatorname{Ind}_{ap}(X)\subseteq \AP^{[d]}(X)$, for every $d\geq 1$ being an integer.

The following characterization of $\operatorname{Ind}_{ap}(X)$ was proved for $\Z$-actions, but its proof also works for abelian group actions.

\begin{lemma}[{\cite[Lemma 7.6]{Cai_Shao_top_charact_independence:2022}}]\label{lemma: characterization_Ind_ap}
    Let $(X,T)$ be a system and $(x_{1},x_{2})\in X\times X\setminus \Delta(X)$. Then $(x_{1},x_{2})\in \operatorname{Ind}_{ap}(X)$ if and only if for every integer $d\geq 1$, $\{x_{1},x_{2}\}^{d} \subseteq N_{d}(X)$.
\end{lemma}

By \cref{prop: characterization_minimal_in_RPd} and \cref{lemma: characterization_Ind_ap}, we obtain the following result.

\begin{proposition}\label{prop: minimal_points_RPinfty_iff_Indap}
    Let $(X,T)$ be a minimal system. If $(x_{1},x_{2})\in X\times X$ is a minimal point, then $(x_{1},x_{2}) \in \RP^{[\infty]}(X)$ if and only if $(x_{1},x_{2})\in \operatorname{Ind}_{ap}(X)$.
\end{proposition}

By \cref{prop: minimal_points_RPinfty_iff_Indap}, one can prove that every minimal system without nontrivial $\operatorname{Ind}_{ap}$-pairs is a proximal extension of its maximal factor of order $\infty$. The following theorem gives a refinement of this result.

\begin{theorem}\label{thm: non_Indap}
    Let $(X,T)$ be a minimal metric system without nontrivial $\operatorname{Ind}_{ap}$-pairs. Then $(X,T)$ is an almost one-to-one extension of its maximal factor of order $\infty$.
\end{theorem}

\begin{proof}
    Let $\pi:X\to X_{\infty}$ be the factor map onto its maximal factor of order $\infty$. By \cref{lemma: dense_preimage_in_Nd_infty}, there exists a dense $G_{\delta}$ set $X_{0}$ of $X$ such that, for every integer $d\geq 1$ and every $x\in X_{0}$,  $(\pi^{(d)})^{-1}(\pi^{(d)}(x^{(d)})) \subseteq N_{d}(X)$. 
    
    Let $x\in X_{0}$ and $y\in \pi^{-1}(x)$. Then \begin{align*}
        \{x,y\}^{d} \subseteq (\pi^{(d)})^{-1}(\pi^{(d)}(x^{(d)})) \subseteq N_{d}(X)
    \end{align*}

    for every integer $d\geq 1$. 
    
    It follows from \cref{lemma: characterization_Ind_ap} that $x=y$, and hence $\pi$ is almost one-to-one.
\end{proof}

\end{document}